\documentclass[12pt]{preprint} 

\usepackage[a4paper,innermargin=1.2in,outermargin=1.2in,
bottom=1.5in,marginparwidth=1in,marginparsep
=3mm]{geometry}

\usepackage[T1]{fontenc} 
\usepackage{fourier} 
\usepackage[english]{babel} 
\usepackage{amsmath,amsfonts,amsthm} 
\usepackage{tikz-cd}
\usepackage{lipsum} 
\usepackage{shuffle}
\usepackage{ amssymb }
\usepackage{fancyhdr} 
\usepackage{abstract}
\usepackage{comment}
\usepackage{stmaryrd}
\usepackage{cprotect}
\usepackage{hyperref}
\usepackage{mathrsfs}
\usepackage{mhequ}

\colorlet{darkblue}{blue!90!black}
\colorlet{darkred}{red!90!black}
 \def\mate#1{\comment[darkblue]{MG: #1}}
 \def\harprit#1{\comment[darkred]{HS: #1}}

\def\scal#1{\langle #1 \rangle}
\def\bigscal#1{\big\langle #1 \big\rangle}

\def\harprit#1{\comment[darkred]{HS: #1}}

\newcommand{\C}{\mathbb{C}}
\newcommand{\cE}{\mathcal{E}}

\newcommand{\T}{\mathbb{T}}

\newcommand{\RR}{\mathbb{R}}
\newcommand{\R}{\mathbb{R}}
\newcommand{\ZZ}{\mathbb{Z}}
\newcommand{\Z}{\mathbb{Z}}
\newcommand{\NN}{\mathbb{N}}
\newcommand{\N}{\mathbb{N}}

\newcommand{\cP}{\mathcal{P}}

\newcommand{\cS}{\mathcal{S}}

\renewcommand{\bf}{\mathbf{f}}

\newcommand{\TT}{\mathbb{T}}
\newcommand{\eps}{\varepsilon}
\newcommand{\cF}{\mathcal{F}}
\newcommand{\E}{\mathbb{E}}
\newcommand{\cB}{\mathcal{B}}
\newcommand{\scF}{\mathscr{F}}

\newcommand{\vertiii}[1]{{\left\vert\kern-0.25ex\left\vert\kern-0.25ex\left\vert #1 
		\right\vert\kern-0.25ex\right\vert\kern-0.25ex\right\vert}}

\newtheorem{theorem}{Theorem}[section]
\newtheorem{corollary}[theorem]{Corollary}

\newtheorem{lemma}[theorem]{Lemma}

\newtheorem{prop}[theorem]{Proposition}

\theoremstyle{remark}
\newtheorem{remark}[theorem]{Remark}

\DeclareMathOperator{\supp}{supp}
\DeclareMathOperator{\Poly}{P}
\newcommand{\id}{\mathrm{id}}
\numberwithin{equation}{section} 
\numberwithin{figure}{section} 
\numberwithin{table}{section} 

\title{	
\normalfont \normalsize 
\large Strong convergence of parabolic rate $1$ of discretisations of stochastic Allen-Cahn-type equations
 \\ 
}
\author{M\'at\'e Gerencs\'er$\,^1$ and Harprit Singh$\,^2$}
\institute{$\,^1$TU Wien, $\,^2$Imperial College London}

\begin{document}

\maketitle 
\begin{abstract}

Consider the approximation of stochastic Allen-Cahn-type equations (i.e. $1+1$-dimensional space-time white noise-driven stochastic PDEs with polynomial nonlinearities $F$ such that $F(\pm \infty)=\mp \infty$)
by a fully discrete space-time explicit finite difference scheme.
The consensus in literature, supported by rigorous lower bounds, is that strong convergence rate $1/2$ with respect to the parabolic grid meshsize is expected to be optimal.
We show that one can reach almost sure convergence rate $1$ (and no better) when measuring the error in appropriate negative Besov norms, by temporarily `pretending' that the SPDE is singular.

\end{abstract}

\tableofcontents
\newpage

\section{Introduction}
Consider the $1+1$-dimensional stochastic Allen-Cahn equation
\begin{equ}\label{eq:A-C}
\partial_t u=\Delta u+u-u^3+\xi
\end{equ}
on $[0,\infty)\times\T$, driven by a $1+1$-dimensional space-time white noise $\xi$,
with some initial condition $\psi$.
The well-posedness of \eqref{eq:A-C} is classically well-understood (see e.g.\cite{FJL}).
In this paper we are interested in the discretisation of \eqref{eq:A-C} and SPDEs of the more general form
\begin{equ}\label{eq:general}
\partial_t u=\Delta u+F(u)+\xi.
\end{equ}
When the nonlinearity $F$ is globally Lipschitz continuous, then it is well-known \cite{Gy0}  that with grid of parabolic meshsize of order $n^{-1}$ (for precise details see below) the error of a space-time explicit finite difference scheme is of order $n^{-1/2}$ in $L_p(\Omega)$.
The results of \cite{Gy0} have seen far-reaching extensions both in terms of the regularity and the growth conditions imposed on $F$.
As for regularity, \cite{Mate} proved rate $1/2$ in the case of merely bounded measurable $F$, without assuming any continuity.
As for growth, \cite{Jentzen} showed rate $1/2$ (for a different, Galerkin-type fully discrete scheme)
for a class of $F$-s that are only locally Lipschitz, in particular covering the example \eqref{eq:A-C}.
Further recent results on full discretisations on \eqref{eq:A-C} and \eqref{eq:general} can be found in \cite{Liu, Wang}.
Let us also mention that for Burgers-type equations, where $F$ depends also on the gradient of the solution, the rate of convergence $1/2$ has been proven for spatial semidiscretisations \cite{AG, HMat}.
The appearance of the exponent $1/2$ everywhere above is not a coincidence, \cite{DG} shows that this is sharp in the following sense: 
even in the simplest linear case $F\equiv 0$ the conditional variance of $u_1(0)$ given the discrete observations of $\xi$ is bounded from below by a positive constant times $n^{-1/2}$.
A similar lower bound is derived in \cite{Jentzen} for the Galerkin-type approximations considered therein.

With matching lower and upper bounds, there appears to be not much room for improvement. The present paper aims to show otherwise.
As already observed by Davie and Gaines \cite{DG} in the simplest case $F\equiv 0$, although the pointwise error is of order $n^{-1/2}$, this does not rule out superior convergence rates when measured in a distributional norm. 
First we make this observation quantitative: measuring in the Besov space $B^{\alpha}_{\infty,\infty}$, $\alpha<1/2$, the error in the linear case $F\equiv 0$ is shown to be of order $n^{(-1/2+\alpha+\eps)\vee(-1)}$, for any $\eps>0$, see Lemma \ref{lem:upperbound} below. 
Pursuing this idea further for a nonlinear equation like \eqref{eq:A-C} and considering $u$ and its approximation $u^n$ as elements of $B^{\alpha}_{\infty,\infty}$ comes across an obvious obstacle: as soon as $\alpha<0$, the mapping $u\mapsto u^3$ is simply not defined. This is reminiscent to the difficulty that one has to overcome when solving higher dimensional $\Phi^4_d$ equations.

In this way, the SPDE looks singular in the sense of \cite{H0}, albeit in an artificial way, and one does not actually expect any renormalisation to appear.
In fact, we will only need the simplest tool for singular SPDE-s, the Da Prato-Debussche trick \cite{DPD}.
It turns out that for any polynomial $F$ of odd degree with negative leading order coefficient this is sufficient to implement the above strategy and we obtain rate $1-\eps$ of strong convergence, when measuring the error in the appropriate Besov norm.
This is the content of our main result, Theorem \ref{thm:main} below, which we state after the precise formulation of our setup.

\begin{remark}
The rate $1$ can not be improved even in a distributional sense, see Proposition \ref{prop:lowerbound}.
\end{remark}

\begin{remark}\label{rem:BDG}
In the proof
one surprisingly encounters seemingly unrelated regularisation by noise tools.
The precise occurence of this term is $\cE^{n,1}$ in \eqref{eq:CE-decomp} in the proof of the main result, here we just outline the heuristics of such a term.
One term that one expects to see when comparing the true equation \eqref{eq:general} with its approximation is $F(u)-F(\pi_n u)$, where $\pi_n$ is the projection on the space-time grid. Since the parabolic regularity of the solution $u$ is known to be $1/2$, it may look like such a term can also not be bounded by anything better than $n^{-1/2}$.
One way of improving the rate for such terms is to notice that since in the mild formulation they appear under a space-time integral,
one can exploit averaging effects.
Using estimates obtained through regularisation by noise, one can get a bound of order $n^{-1}$ \cite{Mate}.
\end{remark}
\begin{remark}\label{rem:SDE}
There is no clear finite dimensional analogy to our result. Again staying on the heuristic level: suppose that one wishes to improve on the well-known rate $1/2$ for a finite dimensional SDE
\begin{equ}\label{eq:SDE-rem}
dX_t=\sigma(t,X_t)\,dW_t,
\end{equ}
where $W$ is a standard Brownian motion, by measuring the error in some negative Besov (in time) norm. Under mild assumptions on $\sigma$ it is known \cite{MG2004} that any approximation method $X^n$ based on $n$ evaluations of $W$ makes an error of order $n^{-1/2}$ at time $1$. Then one expects that any approximation on the time interval $[1,2]$, measured in any norm, would make a larger strong error than the exact solution $Y$ of \eqref{eq:SDE-rem} starting at time $1$ from $Y_1=X^n_1$.
But choosing $\sigma(t,x)\equiv 1$ for $t\in[1,2]$, the difference between the true solution $X$ and $Y$ is constant in time and equals $X_1-Y_1=X_1-X^n_1$, and therefore cannot be made of higher order in any reasonable norm. See also Figure \ref{figureB} below.
\end{remark}

\subsection{The setup}
Fix a complete probability space $(\Omega,\cF,\mathbb{P})$ carrying a $1+1$-dimensional space-time white noise $\xi$. That is, $\xi$ is a mapping from $\cB_b([0,\infty)\times \T)$, the bounded Borel sets of $[0,\infty )\times \T$, to $L_2(\Omega)$ such that for any collection $A_1,\ldots,A_k$ of elements of $\cB_b([0,\infty)\times\T)$, the vector $\big(\xi(A_1),\ldots,\xi(A_k)\big)$ is Gaussian with mean $0$ and covariance $\E\big(\xi(A_i)\xi(A_j)\big)=|A_i\cap A_j|$.
We denote by $(\cF_t)_{t\geq0}$ the complete filtration generated by $\xi$.

We take the nonlinearity $F$ in \eqref{eq:general} to be a polynomial of odd degree with negative leading order coefficient. That is, 
\begin{equ}
F(v)=\sum_{j=0}^{\nu} c_j v^j,
\end{equ}
with some odd integer $\nu\geq 3$, $c_0,\ldots, c_{\nu-1}\in \RR$, and $c_{\nu}<0$.
The prototypical example to have in mind is the stochastic Allen-Cahn equation \eqref{eq:A-C}, where $F(v)=-v^3+v$.

As in \cite{Gy0, Mate}, we consider a finite difference, forward Euler approximation of \eqref{eq:general}. Introduce the space and time grids
\begin{align}
\Pi_n &=\{0,(2n)^{-1},\ldots,(2n-1)(2n)^{-1}\}\subset \mathbb{T},\qquad
\Lambda_n= \{0, h, 2h,\ldots\}\subset \RR,
\end{align}
for $n\in\NN$ and where $h$ satisfies the relation $h=c(2n)^{-2}$ for some $c\in(0,1/4)$.
On $\Pi_n$, just like on $\TT$,  addition (as well as negation and subtraction) is understood in a periodic way, e.g. $(2n-1) (2n)^{-1}+(2n)^{-1}=0$.

The approximation scheme is defined by setting
$u^n_0( x)=\psi( x)$ for $ x\in\Pi_n$ and
then inductively
\begin{align}
u^n_{{t}+h}( {x}) &=u^n_{{t}}( {x})+h \Delta_n u^n_{ t}( x)
+h F(u^n_{ t}({x}))+h  \eta_n ({t}, {x}) \label{eq:approx classical form}
\end{align}
for $ t\in\Lambda_n$ and $ x\in\Pi_n$, where
the discrete Laplacian is defined as
\begin{equation}
\Delta_n f( x)=\frac{f( x+(2n)^{-1})-2f( x)+f( x-(2n)^{-1})}{(2n)^2}\, ,
\end{equation}
the discrete noise term is given by
\begin{equation}
\eta_n ({t}, {x})=2n h^{-1}\xi\Big([ t, t+h]\times[ x, x+(2n)^{-1}]\Big)\ .
\end{equation}

\subsection{Function spaces}
The function spaces $C^\alpha(\TT)$ for $\alpha\in (0,1)$, $L^p(\Omega)$, $L^p(\TT)$ for $p\in[1,\infty]$ are defined in the usual way. By $C(\TT)$ we denote the the space of continuous functions on $\TT$ equipped with the supremum norm. For $k\in\N$, $C^r(\T)$ ($C^r(\R)$, resp.) denote the $r$ times continuously differentiable functions equipped with the usual norm.
If the target space differs from $\RR$ (for example, $\C$) it is indicated in the notation (for example, as $L^2(\TT;\C)$).
By $\scal{\cdot,\cdot}$ we denote the complex inner product on $L^2(\T;\C)$.

For functions on $\Pi_n$ we define analogous norms 
\begin{equs}
\|f\|_{C^\alpha(\Pi_n)}&=\sup_{x\in \Pi_n}|f(x)|+[f]_{C^\alpha(\Pi_n)}=\sup_{x\in \Pi_n}|f(x)|+\sup_{x\neq y\in\Pi_n}\frac{|f(x)-f(y)|}{|x-y|^\alpha},
\\
\|f\|_{L^p(\Pi_n)}^p&=\frac{1}{2n}\sum_{x\in\Pi_n}|f(x)|^p,
\end{equs}
for $\alpha\in(0,1)$ and $p\in [1,\infty)$.
As usual, the norm $\|\cdot\|_{L^2(\Pi_n;\C)}$ induces a complex inner product that we denote by $\langle\cdot,\cdot\rangle_{n}$.
By $C(\Pi_n)$ we denote the space of all functions on $\Pi_n$ and by $\|f\|_{L^\infty(\Pi_n)}$ the maximum norm on $C(\Pi_n)$.
Convolution of elements of $C(\Pi_n)$ is denoted by $\ast_n$, that is, $(f\ast_n g)(x)=\scal{f(x-\cdot),g(\cdot)}_n$.
The norms $\|\cdot\|_{L^p(\Pi_n)}$ satisfy H\"older's inequalities. In addition, one has the $n$-dependent inequality $\|f\|_{L^\infty(\Pi_n)}\leq 2n\|f\|_{L^1(\Pi_n)}$. More generally, for any $1\leq q\leq p\leq \infty$ one has
\begin{equ}\label{eq:reverse-discrete-Lp}
\|f\|_{L^p(\Pi_n)}\leq (2 n)^{1/q-1/p}\|f\|_{L^q(\Pi_n)},
\end{equ}
see Lemma \ref{lem:bernstein} below for a slightly stronger version and a short proof.

Any continuous function on $\TT$ can be restricted to $\Pi_n$. We denote this operation by
\begin{equ}
\delta: C(\TT)\to C(\Pi_n),\qquad (\delta f)(x)=f(x),\,x\in\Pi_n.
\end{equ}

The functions $e_j(x)= e^{i 2 \pi j x}$ for $j\in\ZZ$ are eigenfunctions of $\Delta$ with eigenvalues $\lambda_j= -4 \pi^2j^2$.
It is also well-known that $(e_j)_{j \in \ZZ}$ forms an orthonormal basis of $L^2(\TT; \mathbb{C})$. 
From now on,  the conjugate of a complex number $z\in\mathbb{C}$ is denoted by $\overline{z}$.

Recall from \cite[Prop.~2.1.1]{Mate} that $\delta e_j$, $j\in\ZZ$ are eigenfunctions of $\Delta_n$, with eigenvalues
\begin{equ}
\lambda^n_j=-16 n^2 \sin^2\Big( \frac{j \pi}{2n}\Big).
\end{equ}
Moreover, the set $(\delta e_j)_{j=-n,-n+1,\cdots ,n-1}$ forms an orthonormal basis of $L^2(\Pi_n;\C)$.
For $j\in\{-n,\ldots,n-1\}$, the ratios of the eigenvalues $\gamma_j^n=\frac{\lambda_j^n}{\lambda_j}=\frac{\sin^2(j\pi/2n)}{(j\pi/2n)^2}$, with the convention $\gamma_0^n=1$, satisfy the bounds
\begin{equ}\label{eq:gamma-prop1}
4 \pi^{-2} \leq \gamma^n_j\leq 1,
\end{equ}
\begin{equs}        \label{eq:gamma-prop2}
|1- \gamma^n_j| \leq \frac{1}{3}\Big( \frac{j \pi}{n } \Big)^2,
\end{equs} 
see \cite{Mate}.

\begin{remark}\label{rem:symm1}
The lack of symmetry in the range of $j$ stems from the fact that we have an even number of grid points. But note that $\delta e_{-n}=\delta e_n$ and thus $\delta e_n=\delta\big((1/2)(e_{-n}+e_n)\big)$.
\end{remark}

Next, we define Besov spaces.
Our setup will be a bit more convoluted than usual, in order to be able to handle their discrete counterparts conveniently.
Fix some $\eps_0\in(0,1/10)$ and take a smooth even bump function $\phi^0:\mathbb{R}\to [0,1]$, such that $\phi^0|_{B_{1-\eps_0}}\equiv 1$ and $\supp \phi^0\subset B_{1}$, where $B_r=\{x\in\R:\,|x|\leq r\}$.
Set $\rho_0=\frac{3-2\eps_0}{2}$.
For $\rho\in[1,2]$ define $\phi_\rho^0(x)=\phi^0\big((\rho_0/\rho)x\big)$.
Further, set for positive integers $j$
\begin{equ}
\phi^{j}_\rho(x)= \phi^{0}_\rho(2^{-j}x)-\phi^{0}_\rho(2^{-j+1}x)=\phi_\rho^1(2^{-j+1}x)\ .
\end{equ}
Note that the definitions are set up such that $\phi^1_\rho\equiv 1$ on $\big(\rho/\rho_0,(\rho/\rho_0)(2(1-\eps_0))\big)$; the midpoint of this interval is exactly $\rho$. Similarly, $\phi^{j+1}_\rho$ is constant $1$ in a neighborhood of $2^j\rho$, this fact is used in Remark \ref{rem:PL} below.

For any distribution $f\in \cS'(\T)$ define the Littlewood-Paley blocks
\begin{equ}
f^{[j]}=\sum_{k\in \Z}\phi^j_1(k)\scal{f, e_k}e_k.
\end{equ}
We then define the (inhomogeneous) Besov spaces via the norms
\begin{equ}
\|f\|_{B^{\alpha}_{p,q}(\T)}=\big\|j\mapsto 2^{\alpha j}\|f^{[j]}\|_{L^p(\T)}\big\|_{\ell^q}.
\end{equ}
So far the choice of $\rho$ did not play a role and in fact it replacing $\phi^j_1$ with $\phi^j_\rho$ would result in an equivalent norm for any $\rho\in[1,2]$. 

For the discrete analogues of Besov norms, for any integer $n\geq 2$ define $J_n=\lfloor \log_2 n\rfloor$ and $\rho_n=n2^{-J_n}\in[1,2]$.
For $f\in C(\Pi_n)$ define the discrete Littlewood-Paley blocks
\begin{equ}
f^{[j],n}=\sum_{k=-n}^{n-1} \phi^j_{\rho_n}(k)\scal{f, \delta e_k}_n\delta e_k
\end{equ}
and the discrete Besov norms
\begin{equ}
\|f\|_{B^{\alpha}_{p,q}(\Pi_n)}=\big\|j\mapsto 2^{\alpha j}\|f^{[j],n}\|_{L^p(\Pi_n)}\big\|_{\ell^q}.
\end{equ}
Just like \eqref{eq:reverse-discrete-Lp} allows one to trade powers of $n$ for integrability, one can also trade for regularity: for any $\alpha\leq \beta$ one has 
\begin{equ}\label{eq:reverse-discrete-regularity}
\|f\|_{B^{\beta}_{p,q}(\Pi_n)}\leq (2n)^{\beta-\alpha}\|f\|_{B^{\alpha}_{p,q}(\Pi_n)}.
\end{equ}
\begin{remark}\label{rem:PL}
With the above setup, we have that for $j\in\{0,\ldots J_n\}$, $\phi^j_{\rho_n}$ is supported on $B_{2^{J_n}\rho_n/\rho_0}=B_{n/\rho_0}$, therefore one has
\begin{equ}
f^{[j],n}=\sum_{k\in\Z} \phi^j_{\rho_n}(k)\scal{f, \delta e_k}_n\delta e_k,
\end{equ}
now importantly having the summation over $k\in\Z$.
For the $(J_n+1)$-th block we can use the fact that $\delta e_k=\delta e_{k-2n}$, to rewrite it as
\begin{equ}
f^{[J_n+1],n}=\sum_{k\in\Z} \tilde\phi^{J_n+1}_{\rho_n}(k)\scal{f, \delta e_k}_n\delta e_k,
\end{equ}
where $\tilde \phi^{j+1}_\rho=\phi^{j+1}_\rho\mathbf{1}_{(0,2^j\rho)}+\phi^{j+1}_\rho(\cdot-2^{j+1}\rho)\mathbf{1}_{[2^j\rho,\infty)}$ is the rescaling of a smooth (!) function $\tilde\phi^1_1$. In particular, $\tilde\phi^{J_n+1}_{\rho_n}(x)=\tilde \phi^1_1(n^{-1}x)$.
\end{remark}

While the restriction operator $\delta$ is rather canonical, for extending functions from $\Pi_n$ to $\T$ there are many choices.
One that plays an important role for us is given by
\begin{equ}
\iota: C(\Pi_n)\to C(\T),\qquad \iota f = \sum_{j=-n}^{n-1}\langle f,\delta e_j\rangle_{n} e_j.
\end{equ} 
One has $\delta\iota=\id_{C(\Pi_n)}$ but $\iota\delta\neq \id_{C(\T)}$ in general.

\begin{remark}\label{rem:symm2}
One observes that even when $f\in C(\Pi_n)$ is real valued, in general $\iota f$ is complex valued. Another natural extension map (cf. Remark \ref{rem:symm1}), which preserves being real valued would be given by 
$$\tilde{ \iota}f:=\sum_{j=-(n-1)}^{n-1}\langle f,\delta e_j\rangle_{n} e_j +  \frac{\langle f,\delta e_{-n}\rangle_{n}}{2} (e_n+e_{-n}).$$
While the maps $\iota$ and $\tilde{ \iota}$ share all relevant properties, the former has a more compact explicit formula and thus we work with it.
\end{remark}

Finally, a notational convention: in proofs of statements we
use the shorthand $f\lesssim g$ to mean that there exists a constant $N$ such that $f\leq N g$, and that $N$ does
not depend on any other parameters than the ones specified in the statement. 
Any additional dependence is denoted in subscript.

%
%

\subsection{Main result}
We are now in a position to state the main result of the article.
\begin{theorem}\label{thm:main}
Let $\theta\in(-1/2,0]$ and $\eps\in(0,1/2+\theta)$. Assume that $\psi\in C^{1-\eps}(\T)$. Then there exists an almost surely finite random variable $\eta$ such that for all $n\in\N$
\begin{equ}\label{eq:main-error}
\sup_{t\in\Lambda_n\cap[0,1]}\|\delta u_t-u^n_t\|_{B^{\theta}_{\infty,\infty}(\Pi_n)}\leq \eta n^{-1/2+\theta+\eps}.
\end{equ}
\end{theorem}
One can rewrite the left-hand side of \eqref{eq:main-error} in terms of the error tested against test functions. For a function $\varphi:\R\to\R$, $x\in\T$, and $\lambda\in(0,1]$ define the rescaled and recentered function $\varphi_x^\lambda(y)=\lambda^{-1}\varphi\big(\lambda^{-1}(y-x)\big)$. Note that if $\varphi$ is smooth on $\R$ and $\supp \varphi\subset B_{1/3}$, then all functions $\varphi_x^\lambda$ are smooth on $\T$.
We define a set of appropriately normalised test functions
\begin{equ}
\Phi=\{\varphi\in C^1(\R):\, \supp \varphi\subset B_{1/3},\, \|\varphi\|_{C^1(\R)}\leq 1 \}.
\end{equ}
The following is then an immediate consequence of Theorem \ref{thm:main} above and Lemma \ref{lemma:testing against test functions} below.
\begin{corollary}
In the setting of Theorem \ref{thm:main}, there exists a random variable $\bar\eta$ such that for all $n\in\N$ 
\begin{equ}
\sup_{t\in\Lambda_n\cap[0,1]}\sup_{\varphi\in \Phi} \sup_{x\in\Pi_n}\sup_{\lambda\in(n^{-1},1]}
\lambda^{-\theta}\Big|\sum_{y\in\Pi_n}\big(u_t(y)-u^n_t(y)\big)(\varphi_x^\lambda(y))\Big|
\leq \bar\eta n^{-1/2+\theta+\eps}.
\end{equ}
\end{corollary}

\begin{remark}
In  comparison to  \cite{Liu, Wang, Jentzen} we do not obtain $L^p(\Omega)$ bounds, on the other hand we require neither an implicit nor a truncated (or ``tamed'') scheme.
This is somewhat reminiscent in spirit of \cite{B19}, where for the stochastic Allen-Cahn equation \eqref{eq:A-C} an explicit time-splitting scheme without taming is studied, and rate of convergence in probability is obtained.
Let us mention however that the key intermediate step of a priori bounds of the approximation can be turned into an $L^p(\Omega)$ bound on a truncated scheme, see Remark \ref{rem:trunc}.
\end{remark}

\begin{remark}
Considering the solutions as elements of distributional spaces is necessary in higher dimensional versions of the stochastic Allen-Cahn equations, also known as the dynamical $\Phi^4_d$ models, which for this reason have to be renormalised.
The dependence of the rate of convergence of approximations on the choice of the Besov exponent is observed in \cite{MZ}, where the authors consider the spatial semidiscretisation of $\Phi^4_2$ and bound the $B^{-\theta}_{\infty,\infty}$ norm of the error by $n^{-\theta+\eps}$ for any $\eps>0$, under the constraint $\theta\in(0,2/9)$. 
In \cite{HMat2,ZZ} the convergence of spatial semidiscretisations of $\Phi^4_3$ is shown, without rate.
\end{remark}

One might wonder if one could even further improve the rate by looking at even weaker norms. The following proposition rules this out, at least for this scheme: even testing the error with a single test function the rate does not exceed $1$.
\begin{prop}\label{prop:lowerbound}
Let $\psi=0$ and $F\equiv 0$. Then for any $\eps>0$ and $t>0$ one has 
\begin{equ}
n^{1+\eps}|\scal{\delta u_t-u^n_t,\delta e_1}_n|\,\,\underset{n\to\infty}{\longrightarrow}\,\,\infty\qquad\text{in probability.}
\end{equ}
\end{prop}

\section{Various tools}

\subsection{Kernels and convolutions}
The continuous and discrete heat kernels are defined as follows.
For $(t,x,y)\in (0,\infty)\times \T^2$, define the kernel 
\begin{equ}
p_t(x,y)
= \sum_{j \in \ZZ} e^{ \lambda_j t } e_j(x-y)\  , 
\end{equ}
and for $t>0$ the corresponding operators
\begin{equ}
f\mapsto \cP_t f,\qquad \cP_t f(x)=\int_\T p_t(x,y)f(y)\,dy,\,\,x\in\T,
\end{equ}
which are known to satisfy the semigroup property $\cP_t(\cP_s f)=\cP_{t+s}f$.
By convention, we understand $p_0(x,y)$ to be the Dirac-$\delta$ distribution, and correspondingly $\cP_0$ to be the identity operator.

Similarly, for $(t,x,y)\in \Lambda_n\times (\Pi_n)^2$, define the discrete kernel
\begin{equ}
p^n_t(x,y)=\sum_{j=-n}^{n-1}(1+h\lambda_j^n)^{t/h} e_j(x-y) \ ,
\end{equ} 
and for $t\in\Lambda_n$ the corresponding operators
\begin{equ}
f\mapsto \cP^n_t f,\qquad \cP^n_t f(x)=\frac{1}{2n}\sum_{y\in\Pi_n}p^n_t(x,y)f(y),\,\,x\in\Pi_n,
\end{equ}
which also satisfy the semigroup property $\cP_t^n(\cP_s^n f)=\cP_{t+s}^nf$, see \cite[Prop.~2.1.6]{Mate}.
Note that the argument of the operator $\cP^n_t$ can either be an element of $C(\T)$ or $C(\Pi_n)$, we do not distinguish these cases in the notation.
One furthermore has that the action of $\cP_t^n$ on the Fourier modes is similarly exponential as that of $\cP_t$: more precisely, there exists a $\kappa=\kappa(c)>0$ such that for all $t\in\Lambda_n$ one has
\begin{equ}\label{eq:exp-discrete}
|1+h\lambda^n_j|^{t/h}\leq e^{-\kappa t j^2},
\end{equ}
see \cite[Eq.~(2.9)]{Mate}.
We rephrase equations \eqref{eq:general} and \eqref{eq:approx classical form} in mild formulations.
Indeed, for \eqref{eq:general} we \emph{define} a mild solution of it as a function $u:\Omega\times[0,\infty)\times\T$ that is almost surely continuous in $t,x$, is measurable with respect to the product of the predictable and Borel $\sigma$-algebras, and satisfies for $(t,x)\in [0,\infty)\times \T$ almost surely
\begin{equ}\label{eq:main in mild form}
u_t(x)=\mathcal{P}_t \psi(x)+\int_0^t\mathcal{P}_{t-s} F (u_s)(x) ds+\int_0^t\int_\T p_{t-s}(x-y)\xi (dy, ds).
\end{equ}
As for the approximation $u^n$, it is already uniquely defined by \eqref{eq:approx classical form}, but one can derive a similar mild form of it: by \cite[Eq.~(2.14)]{Mate} one has for $(t,x)\in \Lambda_n\times \Pi_n$ almost surely
\begin{equ}
u^n_{ t}( x)
=\mathcal{P}^{n}_{ t} \psi( x)+\int_0^{ t} \mathcal{P}^{n}_{\kappa_n( t-s)} F (u^n_{\kappa_n(s)})( x)\,ds
+\int_0^{ t}\int_\T p_{\kappa_n( t-s) }^{n}( x,\rho_n(y))\,\xi (dy, ds), \label{eq:approx in mild form}
\end{equ}
where $\kappa_n(t):=\lfloor th^{-1}\rfloor h$ and $\rho_n(x)=\lfloor x2n\rfloor (2n)^{-1}$.

Important reference objects are the solutions of the linear equation with $0$ initial data, they are denoted by
\begin{equ}
O_t(x)=\int_0^t\int_\T p_{t-s}(x-y)\xi (dy, ds)
\end{equ}
for $(t,x)\in(0,\infty)\times\T$, and by
\begin{equs}
O^n_t(x)
= \int_0^{ t}\int_\T p_{\kappa_n( t-s) }^{n}( x,\rho_n(y))\,\xi (dy, ds) 
\end{equs}
for $(t,x)\in \Lambda_n\times \Pi_n$, respectively.
Unraveling the definitions, one finds 
\begin{equs}
\scal{ O_t, e_\ell} &=\bigscal{ \int_0^t\int_\T \sum_{k \in \ZZ} e^{- 4 \pi^2 k^2 (t-s) } e_k(\cdot) \bar{e}_k(y)\xi (dy, ds), e_\ell}\\
&= \int_0^t\int_\T e^{- 4 \pi^2 \ell^2 (t-s) }  \bar{e}_\ell(y)\xi (dy, ds) \label{eq:O Fourier}
\end{equs}
for any $\ell\in \Z$
and
\begin{equs}
\langle \iota O^n_t, e_\ell\rangle =\langle  O^n_t,\delta e_\ell\rangle_{n}
&=\bigscal{\int_0^t\int_\T \sum_{k=-n}^{n-1}(1+h\lambda_k^n)^{(\kappa_n (t-s)/h)}\delta e_k(\cdot)\bar{e}_k(\rho_n(y))\xi (dy, ds),
\delta e_\ell}_n
\\
&= \int_0^t\int_\T (1+h\lambda_\ell^n)^{(\kappa_n (t-s)/h)} \overline{e}_\ell(\rho_n(y))  \xi (dy, ds)\label{eq:On Fourier}
\end{equs}
for any $\ell\in\{-n,\ldots,n-1\}$. For any other $\ell$, clearly $\langle \iota O^n_t, e_\ell\rangle=0$.

\subsection{Properties of discrete and continuous function spaces}\label{sec:spaces}
In this section we establish properties of the discrete Besov spaces analogous to the continuous ones. 
We start with two lemmata which substantiates the definition of discrete Besov spaces, their proof can be found in the Appendix.
\begin{lemma}\label{lemma:testing against test functions}
For any $\alpha<0$ there exists a constant $N=N(\alpha,\phi^0)$ such that with $r=\min\{k\in \mathbb{N}\ : \ k>-\alpha\}$ the bound
\begin{equ}
\sup_{x\in\Pi_n}\sup_{\lambda\in (n^{-1},1)}\frac{\big|\scal{f,\delta\varphi_x^\lambda}_n\big|}{\lambda^{\alpha}}\leq N \|\varphi\|_{C^{r}(\T)} \|f\|_{B^{\alpha}_{\infty,\infty}(\Pi_n)} 
\end{equ}
holds uniformly over $\varphi\in C_c^{\infty}(\T)$ such that $\supp\varphi\subset B_{1/3}$, $n\in \mathbb{N}$, and $f\in C(\Pi_n)$. 
\end{lemma}

\begin{lemma}\label{lem:besov-holder}
For any $\alpha\in(0,1)$ there exists a constant $N=N(\alpha,\phi^0)$ such that for all $n\in\N$, $f\in C(\Pi_n)$ one has the bounds.
\begin{equ}\label{eq:besov-holder}
\|f\|_{B^{\alpha}_{\infty,\infty}(\Pi_n)}\leq N\|f\|_{C^\alpha(\Pi_n)},
\qquad
\|f\|_{C^\alpha(\Pi_n)}\leq N \|f\|_{B^{\alpha}_{\infty,\infty}(\Pi_n)} \ .
\end{equ}
\end{lemma}

\

Next, recall the following classical Theorem \cite[Thm.~II.7.10]{Zyg03}.
\begin{theorem}\label{theorem classical}
Denote by $\delta_{x}$ the Dirac measure at $x\in \mathbb{R}$ and let $\omega_{m}= \frac{1}{m}\sum_{j=0}^{m-1} \delta_{\frac{ j}{m}}$.  There exists a constant $N$ such that for any complex polynomial $P(z)= \sum_{k=0}^n c_k z^k$ and all $p\in [1,\infty]$
$$\Big(\int_0^{1} |P(e^{2\pi it})|^p d\omega_{n+1}(t) \Big)^{\frac{1}{p}}\leq N\Big(\int_0^{1} |P(e^{2\pi it})|^p dt \Big)^{\frac{1}{p}} \ .$$
Similarly, for each $p\in (1,\infty)$, there exists a constant $N=N(p)$ such that
$$\Big(\int_0^{1} |P(e^{2\pi it})|^p dt \Big)^{\frac{1}{p}}\leq N\Big(\int_0^{1} |P(e^{2\pi it})|^p d\omega_{n+1}(t) \Big)^{\frac{1}{p}}$$ uniformly over polynomials of the form $P(z)= \sum_{k=0}^n c_k z^k$ .
\end{theorem}
A simple consequence of Theorem~\ref{theorem classical} is a kind the equivalence of discrete and continuous $L^p$ norms (which can be used to deduce equivalence of Besov norms, but we do not need this).
\begin{prop}\label{prop lp equivalence}
For any $p\in (1,\infty)$ the maps $\iota: L^p(\Pi_n) \to L^p (\T)$ are bounded uniformly in $n\in\N$. Furthermore, the maps $\delta: \iota L^p (\T) \to L^p(\Pi_n)$ are bounded uniformly in $n\in \mathbb{N}$ and $p\in [1, \infty]$.
\end{prop}
\begin{proof}
It suffices to apply Theorem \ref{theorem classical} to the polynomial $P(z)=\sum_{k=0}^{2n-1} \langle f, \delta e_{k-n}\rangle_n z^k$
since one has
\begin{equs}
|f(x)|&= \Big|\sum_{k=-n}^{n-1} \langle f, \delta e_k\rangle_n \delta e_k (x)\Big|
\\
&= \Big|\sum_{k=0}^{2n-1} \langle f, \delta e_{k-n}\rangle_n \delta e_{k-n}(x)\Big|
\\
&=\Big|\sum_{k=0}^{2n-1} \langle f, \delta e_{k-n}\rangle_n (\delta e_1)^k\Big|= \big|P(e^{2\pi i x})\big|
\end{equs}
for $x\in \Pi_n$  and similarly $|\iota f(x)|=  |P(e^{2\pi i x})|$ for $x\in \T$ \ .
\end{proof}

\begin{remark}
Note that in the case $p=2$, the above proposition follows immediately from the fact that $\{\delta e_k\}_{k=-n,\ldots,n-1}$ ($\{e_k\}_{k\in\Z}$, resp.) is orthonormal basis of $L^2(\Pi_n;\C)$ ($L^2(\T;\C)$, resp.). 
\end{remark}

\begin{prop}\label{prop:equivalence}
There exists an constant $N$ (depending only on the choice of $\phi^0$) such that for all $p\in[1,\infty]$, $n\in\N$, $f\in C(\Pi_n)$ one has the bound 
\begin{equ}\label{first}
\|f\|_{B^{0}_{p,\infty}(\Pi_n)} \leq N \|f\|_{L^p(\Pi_n)},\qquad \|f\|_{L^p(\Pi_n)}\leq \|f\|_{B^{0}_{p,1}(\Pi_n)}.
\end{equ}
Further, for any $\eps>0$ there exists $N=N(\eps)$ such that for all $\alpha\in\R$, $p\in[1,\infty]$, and $f\in C(\Pi_n)$ one has
\begin{equ}\label{second}
\|f\|_{B^{\alpha}_{p,1}(\Pi_n)}  \leq N \|f\|_{B^{\alpha+\eps}_{p,\infty}(\Pi_n)}.
\end{equ}
\end{prop}

Before proceeding to the proof, we make the following useful observation, used several times in the sequel.
\begin{remark}\label{rem:fourier}
Denote the usual Fourier transform by
\begin{equ}
\scF(f)(z)=\int_\R f(x)e^{2\pi i z x}\,dx.
\end{equ}
Then (as one can see from e.g. the Poisson summation formula) 
if $f\in L^1(\mathbb{R})$ is a smooth function on $\R$ and $g$ is a smooth function on $\T$, such that 
$\mathcal{F}f(k)= \langle g, e_k \rangle$
for all $k\in\Z$, then one has $g(x)= \sum_{n\in \mathbb{Z}} f(x + n)$ on $\T$. 
In particular for any $\alpha\geq 0$,
$$\big\||\cdot|^\alpha g(\cdot) \big\|_{L^1(\T) }\leq \big\| |\cdot|^\alpha f(\cdot)\big\|_{L^1(\mathbb{R})}\, $$
and
$$\big\||\cdot|^\alpha g(\cdot) \big\|_{L^1(\Pi_n) }\leq \big\| |\cdot|^\alpha f(\cdot)\big\|_{L^1((2n)^{-1}\mathbb{Z})}\ ,$$
where here, and in the sequel, for $\gamma>0$ the space $\gamma\mathbb{Z}$ is understood to be equipped with the measure $\frac{1}{\gamma}\sum_{x\in \gamma\mathbb{Z}}  \delta_x$.
\end{remark}

\begin{proof}[Proof of Proposition \ref{prop lp equivalence}.]
First we write
\begin{equ}\label{aaaa}
f^{[j],n}= \sum_{k\in\Z} \hat\phi^j_{\rho_n}(k)\langle f, \delta e_k\rangle_n \delta e_k =  \Big( \sum_{k\in\Z} \hat\phi^j_{\rho_n}(k)\delta e_k \Big)\ast_n f ,
\end{equ}
where $\hat \phi^j_{\rho_n}=\phi^j_{\rho_n}$ for $j\neq J_n+1$ and $\hat \phi^{J_n+1}_{\rho_n}=\tilde \phi^{J_n+1}_{\rho_n}$, see Remark \ref{rem:PL}.
Thus the first inequality of \eqref{first} follows by Young's convolution inequality and the fact that by Proposition \ref{prop lp equivalence} and Remark \ref{rem:fourier} we have the following uniform (in $n$ and $j$) $L^1$ bound:
\begin{equ}\label{eq:disc-to-cont-fourier}
\big\|\sum_{k\in\Z} \hat\phi^j_{\rho_n}(k)\delta e_k\big\|_{L^1(\Pi_n)}\leq \big\|\sum_{k\in\Z} \hat\phi^j_{\rho_n}(k)e_k\big\|_{L^1(\T)}\leq  \| \mathcal{F}^{-1} (\hat\phi^j_{\rho_n}) \|_{L^1(\mathbb{R})}= \| \mathcal{F}^{-1} (\hat\phi^1_1) \|_{L^1(\mathbb{R})}.
\end{equ}
The second inequality in \eqref{first} follows simply from the triangle inequality, since $f= \sum_{j\in\N} f^{[j],n}$.
The inequality \eqref{second} follows by observing that for any sequence $\{a_j\}_{j\in \mathbb{N}}$ one has
$$ \sum_{j\in\N} 2^{\alpha j} |a_j| \lesssim_\eps \sup_{j\in\N} 2^{(\alpha+\eps) j} |a_j|\ .$$
\end{proof}

\begin{lemma}\label{lem:discrete-Besov-embedding}
Let $1\leq p_1\leq p_2\leq\infty$, $1\leq q_1\leq q_2\leq \infty$, and $\alpha\in \R$.
Then there exists a constant $N=N(p_1,p_2)$ such that for all $f\in C(\Pi_n)$ one has the bound
\begin{equ}
\|f\|_{B^{\alpha-(1/p_1-1/p_2)}_{p_2,q_2}(\Pi_n)}\leq N\|f\|_{B^{\alpha}_{p_1,q_1}(\Pi_n)}.
\end{equ}
\end{lemma}
\begin{proof}
Using the discrete Berstein inequality, Lemma~\ref{lem:bernstein} below, we find that $$2^{\alpha j-(1/p_1-1/p_2)j} \| f^{[j],n}\|_{{L^{p_2}}{(\Pi_n)}}\lesssim 2^{\alpha j} \| f^{[j],n}\|_{{L^{p_1}}{(\Pi_n)}} .$$
Observing that $\| \cdot\|_{\ell^{q_2}} \leq \| \cdot\|_{\ell^{q_1}} $ concludes the proof.
\end{proof}

\begin{lemma}[Discrete Bernstein Inequality]\label{lem:bernstein}
For any $1\leq p\leq q\leq \infty$, and any function $f\in C(\Pi_n)$ of the form $f= \sum_{k=-m}^{m-1} a_k \delta e_k$ , where $m \leq n$, the following inequality holds
$$ \|f \|_{L^q(\Pi_n)} \leq (2m)^{\frac{1}{p}- \frac{1}{q}} \|f \|_{L^p(\Pi_n)} \ .$$
\end{lemma}
\begin{proof}
The case $p=q$ is obvious. The case $q= \infty$, $p=1$ follows from
$$\|f \|_{L^\infty(\Pi_n)} \leq \sum_{k=-m}^{m-1}|\langle f, \delta e_k\rangle_n | 
\leq 2m \| f\|_{L^1(\Pi_n)}.$$
The general case follows by interpolation.
\end{proof}

\begin{lemma}\label{lem:product}
Let $\alpha,\beta\in\R$ such that $\alpha + \beta >0$. Then there exists a constant $N=N(\alpha,\beta)$ such that for any functions $f,g\in C(\Pi_n)$  one has the following estimate 
\begin{equ}
\|fg\|_{B^{\alpha\wedge \beta}_{\infty,\infty}(\Pi_n)}\leq N\|f\|_{B^{\alpha}_{\infty,\infty}(\Pi_n)}\|g\|_{B^{ \beta}_{\infty,\infty}(\Pi_n)}.
\end{equ}
\end{lemma}
\begin{proof}
The proof is essentially the same as in the continuous case, for the convenience of the reader we sketch it.
We introduce the quantities $f^{[\leq j],n}= \sum_{i\leq j} f^{[i],n}$ and similarly $g^{[\leq j],n}= \sum_{i\leq j} g^{[i],n}$, as well as the Coifmann-Meyer operators
$$\pi_{<} (f,g):= \sum_{k\in\N}  f^{[k],n} g^{[\leq k-2],n}, \quad \pi_{=} (f,g):= \sum_{\substack{k,\ell\in\N\\|k-\ell|\leq 1}}  f^{[k],n}  g^{[\ell],n} ,\quad
\pi_{>} (f,g):= \pi_{<} (g,f)\ .$$
Trivially, one has 
$$fg= \pi_{<} (f,g)+ \pi_{=} (f,g) +\pi_{>} (f,g), $$
and we shall bound each of these terms separately. Without loss of generality, assume $\|f\|_{B^{\alpha}_{\infty,\infty}(\Pi_n)}=\|g\|_{B^{ \beta}_{\infty,\infty}(\Pi_n)}=1$ and note that then
$$
\|f^{[\leq k-2],n}\|_{L^\infty(\Pi_n)}\leq \sum_{\substack{j\in\N\\j\leq k-2}} \|f^{[j],n}\|_{L^\infty(\Pi_n)}\leq \sum_{\substack{j\in\N\\j\leq k-2}} 2^{-\alpha j} \lesssim 2^{-\alpha k}\vee 1,
$$
and similarly $\|g^{[\leq k-2],n}\|_{L^\infty(\Pi_n)}\lesssim 2^{-\beta k}\vee 1$.
Note further that
since the product $f^{[k],n} g^{[\leq k-2],n}$ only consists of frequencies whose order of magnitude is of bounded distance from $k$ (with our setup of Littlewood-Paley blocks, with distance up to $2$),
the contribution to frequencies of order $2^j$ to $\pi_<(f,g)$ comes only from terms with $|k-j|\leq 2$.Therefore,
\begin{align*}
\big\|\big(\pi_{<} (f,g)\big)^{[j],n}\big\|_{L^\infty(\Pi_n)} & \lesssim \sum_{\substack{k\in\N\\|k-j|\leq 2}} \|f^{[k],n}\|_{L^\infty(\Pi_n)} \|g^{[\leq k-2],n}\|_{L^\infty(\Pi_n)}\\
& \lesssim \sum_{\substack{k\in\N\\|k-j|\leq 2}}  2^{-\alpha k}(2^{-\beta k}\vee 1) \lesssim 2^{-j(\alpha+\beta\wedge 0)}.
\end{align*}
Exactly the same way $\big\|\big(\pi_{>} (f,g)\big)^{[j],n}\big\|_{L^\infty(\Pi_n)} \lesssim 2^{-j(\alpha\wedge0+\beta)}$.
Lastly, for $\pi_=(f,g)$, with a similar reasoning as above, one sees that the contribution to frequencies of order $2^j$ comes from the terms $k\geq j-2$. Therefore
\begin{align*}
\big\|\big(\pi_{=} (f,g)\big)^{[j],n}\big\|_{L^\infty(\Pi_n)} & \lesssim  \sum_{\substack{k,\ell\in\N\\ k,\ell\geq j-2 \\|k-\ell|\leq 1}} \| f^{[k],n}\|_{L^\infty(\Pi_n)} \|  g^{[\ell],n}\|_{L^\infty(\Pi_n)} \\
&\lesssim \sum_{\substack{k,\ell\in\N\\ k,\ell\geq j-2 \\|k-\ell|\leq 1}}  2^{-\alpha \ell- \beta k } \lesssim 2^{-j(\alpha+ \beta )} 
\end{align*}
using the condition $\alpha+\beta>0$ in the very last step. Since the minimum of the exponents $\alpha\wedge0+\beta$ and $\alpha+\beta\wedge 0$ is simply $\alpha\wedge\beta$, we get the claim.
\end{proof}

\subsection{Heat kernel bounds}
Since one can view $p_n(x,y)$ the transition probability from $x$ to $y$ of a random walk on $\Pi_n$ (see \cite[Rem.~2.1.7]{Mate}), one has $\|p^n_t(x,\cdot)\|_{L^1(\Pi_n)}=1$, and therefore $\|\cP^n_tf\|_{L^p(\Pi_n)}\leq\|f\|_{L^p(\Pi_n)}$ for any $p\in[1,\infty]$.
\begin{lemma}\label{lem:discrete-semigroup1}
Let $\alpha\in\R$, $p,q\in[1,\infty]$, and $\delta>0$. Then there exists a constant $N=N(\delta,c)$ such that for all $t\in\Lambda_n\cap[0,1]$ and $f\in C(\Pi_n)$ one has the bound
\begin{equ}
\|\cP^n_t f\|_{B^{\alpha+\delta}_{p,q}(\Pi_n)}\leq Nt^{-\delta/2}\|f\|_{B^{\alpha}_{p,q}(\Pi_n)}.
\end{equ}
\end{lemma}
\begin{proof}
As in the proof of Proposition \ref{prop:equivalence}, we want to appeal to arguments on continuous Fourier transforms.
To this end, first we extend some of our functions: define, for $z\in\R$, $n\in\N$, $t\in\Lambda_n$:
\begin{equ}
\tilde\lambda^n(z)=-16n^2\sin^2\Big(\frac{z\pi}{2n}\Big), \qquad
\tilde\mu_{t,n}(z)=\big(1+h\tilde\lambda^n(z)\big)^{t/h}.
\end{equ}
In particular, we have $\tilde\lambda^n(j)=\lambda^n_j$ for $j\in\Z$.
Tedious but elementary calculations show
\begin{equs}
\partial_z\tilde\mu_{t,n}(z)&=-\frac{8\pi tn\sin\big(\frac{z\pi}{n}\big)}{1+h\tilde\lambda^n(z)}\tilde\mu_{t,n}(z),
\\
(\partial_z)^2\tilde\mu_{t,n}(z)&=\frac{(t^2-ht)\big(8n\pi\sin\big(\frac{z\pi}{n}\big)\big)^2}{(1+h\tilde\lambda^n(z))^2}\tilde\mu_{t,n}(z)-\frac{8\pi^2 t\cos\big(\frac{z\pi}{n}\big)}{1+h\tilde\lambda^n(z)}\big(\tilde\mu_{t,n}(z)\big).
\end{equs}
Recalling that by our choice of $c$ we have $1+h\tilde\lambda^n(z)\geq 1/2$, we get the bounds
\begin{equ}\label{eq:mu-derivatives}
|\partial_z\tilde\mu_{t,n}(z)|\lesssim t|z|\tilde\mu_{t,n}(z),\qquad
|(\partial_z)^2\tilde\mu_{t,n}(z)|\lesssim \big(t^2|z|^2+t)\tilde\mu_{t,n}(z).
\end{equ}

Let $\eta^1_{1}: \R\to [0,1]$ be a smooth compactly supported function such that $\eta^1_1|_{\supp \phi^1_1}\equiv 1$, $0\notin \supp \eta^1_1$, and $\supp\eta^1_1\subset B_{(1+\eps_0)\rho_0^{-1}}$. Then set $\eta^j_\rho(x):= \eta(2^{-j}\rho^{-1}x)$ and observe that for $j\geq 1$
$$
(\cP^n_t f\big)^{[j],n} = \cP^n_t (f^{[j],n}) = \Big(\sum_{k=-n}^{n-1} \eta^j_{\rho_n}(k) \tilde \mu_{t,n}(k)\delta e_k \Big) *_n (f^{[j],n}),
$$ so that it suffices to show 
\begin{equ}\label{usefull estimate}
\big\|\sum_{k=-n}^{n-1} \eta^j_{\rho_n}(k) \tilde \mu_{t,n}(k)\delta e_k \big\|_{L^1(\Pi_n)} \lesssim \big(t^{1/2}2^j\big)^{-\delta}
\end{equ}
 uniformly over $n\in \mathbb{N}$, $t\in \Lambda_n$ and $j\leq J_n+1$.
First consider the $1\leq j\leq J_n$ case. Thanks to our choice of $\eta_1^1$, the support of $\eta^j_{\rho_n}$ is contained in $B_{n}$, so the summation in $k$ can freely be changed to run over $k\in\Z$.
Therefore, first using the argument as in \eqref{eq:disc-to-cont-fourier} we get
\begin{equs}
\big\|\sum_{k\in\Z} \eta^j_{\rho_n}(k) \tilde \mu_{t,n}(k)\delta e_k \big\|_{L^1(\Pi_n)}&\leq \big\|\mathcal{F}^{-1}\big(\eta_{\rho_n}^{1}(\cdot)\tilde \mu_{t,n}(2^j\cdot)\big)\big\|_{L^1(\R)}.
\end{equs}
By H\"olders inequality, we get
\begin{equs}[long1]
\big\|\mathcal{F}^{-1}\big(\eta_{\rho_n}^{1}(\cdot)\tilde\mu_{t,n}(2^j\cdot)\big)\big\|_{L^1(\R)}
&\lesssim \big\|(1+|\cdot|^2)\mathcal{F}^{-1}\big(\eta_{\rho_n}^{1}(\cdot)\tilde\mu_{t,n}(2^j\cdot)\big)\big\|_{L^\infty(\R)}
\\
&\lesssim \big\|\mathcal{F}^{-1}\big((1+\Delta)(\eta_{\rho_n}^{1}(\cdot)\tilde\mu_{t,n}(2^j\cdot))\big)\big\|_{L^\infty(\R)}
\\
&\leq \big\|(1+\Delta)(\eta_{\rho_n}^{1}(\cdot)\tilde\mu_{t,n}(2^j\cdot))\big\|_{L^1(\R)}.
\end{equs}
Using the bounds \eqref{eq:mu-derivatives} we have
\begin{equs}[long2]
\big\|&(1+\Delta)(\eta_{\rho_n}^{1}(\cdot)\tilde\mu_{t,n}(2^j\cdot))\big\|_{L^1(\R)}
\\
&\lesssim \int_{\supp \eta^1_{\rho_n}}|\tilde\mu_{t,n}(2^j z)|+2^j|\partial_z\tilde\mu_{t,n}(2^j z)|+2^{2j}|(\partial_z)^2\tilde\mu_{t,n}(2^j z)|\,dz
\\
&\lesssim\int_{\supp \eta^1_{\rho_n}}\big(1+2^{2j}t|z|+2^{4j}t^2|z|^2+2^{2j}t\big)|\tilde\mu_{t,n}(2^j z)|\,dz
\\
&=\big(t^{1/2}2^j\big)^{-\delta}\int_{\supp \eta^1_{\rho_n}}(t^{1/2}2^j\big)^{\delta}\big(1+2^{2j}t|z|+2^{4j}t^2|z|^2+2^{2j}t\big)|\tilde\mu_{t,n}(2^j z)|\,dz
\\
&\leq (t^{1/2}2^j\big)^{-\delta}\sup_{\substack{z\in \supp\eta^1_{\rho_n}\\ r\geq 0}}
\Big(\big(r^\delta(1+r^2|z|+r^4|z|^2+r^2)\big)e^{-\kappa r^2|z|^2}\Big)
\end{equs}
using \eqref{eq:exp-discrete} in the last step. Since the support of $\eta^1_{\rho_n}$ is separated from $0$ uniformly in $n$, the supremum is bounded uniformly in $n$, finishing the proof of \eqref{usefull estimate} in the case $j\neq 0,J_n+1$.

For $j=J_n+1$, we use the $2n$-periodicity in $k$ of $\tilde\mu_{t,n}$ and $\delta e_{\cdot}$ to rewrite the sum on the right-hand side of \eqref{usefull estimate} as
\begin{equ}
\sum_{k=-n}^{n-1} \eta^{J_n+1}_{\rho_n}(k) \tilde \mu_{t,n}(k)\delta e_k=\sum_{k\in\Z}\tilde \eta^{J_n+1}_{\rho_n}(k) \tilde \mu_{t,n}(k)\delta e_k,
\end{equ}
where $\tilde \eta^{j+1}_\rho=\eta^{j+1}_\rho\mathbf{1}_{(0,2^j\rho)}+\eta^{j+1}_\rho(\cdot-2^{j+1}\rho)\mathbf{1}_{[2^j\rho,\infty)}$ is the rescaling of the smooth (!) function $\tilde\eta_1^1$. In particular, $\tilde\eta^{J_{n}+1}_{\rho_n}(x)=\tilde\eta^1_1(n^{-1}x)$. From there the bound \eqref{usefull estimate} follows as before.

Finally, for $j=1$ we can simply use that the $L^1(\Pi_n)$ norm of the lowest order multiplier is uniformly bounded in $n$ and so the analogue of \eqref{usefull estimate} follows from the trivial inequality $1\leq t^{-\delta/2}$ for $t\in (0,1]$. 
\end{proof}

\begin{lemma}\label{lem:discrete-semigroup2}
Let $\alpha\in\R$, $p,q\in[1,\infty]$, and $\delta\in[0,2]$. Then there exists a constant $N=N(\delta,c)$ such that for all $t\in\Lambda_n\cap[0,1]$ and $f\in C(\Pi_n)$ one has the bound
\begin{equ}
\|\cP^n_t f-f\|_{B^{\alpha}_{p,q}(\Pi_n)}\leq Nt^{\delta/2}\| f\|_{B^{\alpha+\delta}_{p,q}(\Pi_n)}.
\end{equ}
\end{lemma}
\begin{proof}
Defining $\eta^j_\rho$ exactly as in the proof of Lemma~\ref{lem:discrete-semigroup1} we write for $j\geq 1$
$$(\cP^n_t f-f\big)^{[j],n} = \Big(\sum_{k=-n}^{n-1} \eta^j_{\rho_n}(k) (\tilde\mu_{t,n}(k)-1)\delta e_k \Big) *_n (f^{[j],n})\ .$$
Therefore this time we aim for the bound
\begin{equ}\label{useful2}
\big\|\sum_{k=-n}^{n-1} \eta^j_{\rho_n}(k) (\tilde \mu_{t,n}(k)-1)\delta e_k \big\|_{L^1(\Pi_n)} \lesssim \big(t^{1/2}2^j\big)^{\delta}.
\end{equ}
Starting with $1\leq j\leq J_n$, we can again replace the summation in $k$ over all of $\Z$, and following the argument in \eqref{long1}-\eqref{long2} we find
\begin{equs}
\big\|&\sum_{k\in\Z} \eta^j_{\rho_n}(k) (\tilde \mu_{t,n}(k)-1)\delta e_k \big\|_{L^1(\Pi_n)}
\\
&\lesssim\int_{\supp \eta^1_{\rho_n}}
|\tilde\mu_{t,n}(2^j z)-1|+2^j|\partial_z\tilde\mu_{t,n}(2^j z)|+2^{2j}|(\partial_z)^2\tilde\mu_{t,n}(2^j z)|\,dz
\\
&\lesssim\int_{\supp \eta^1_{\rho_n}}
\big(2^{2j}|z|^2t+2^{4j}t^2|z|^2+2^{2j}t\big)e^{-\kappa t 2^{2j}|z|^2}
\\
& \lesssim \big(t^{1/2}2^j\big)^{\delta}\sup_{\substack{z\in \supp\eta^1_{\rho_n}\\ r\geq 0}}
\Big(\big(r^{-\delta}(r^2|z|^2+r^4|z|^2+r^2)\big)e^{-\kappa r^2|z|^2}\Big)
\end{equs}
using $|\tilde\mu_{t,n}(w)-1|\lesssim t|w|^2 e^{-\kappa t|w|^2}$ in the second inequality.
Since $\delta\leq 2$ and the support of $\eta^1_{\rho_n}$ is separated from $0$ uniformly in $n$, the supremum is bounded uniformly in $n$, finishing the proof of \eqref{useful2} in the case $j\neq 0,J_n+1$.

For $j=J_n+1$ we only have to take slight care when replacing the summation in $k$, but this is done precisely as in the proof of Lemma \ref{lem:discrete-semigroup1}.

Finally, for $j=0$, we simply have for any smooth function $\eta$ that is constant $1$ on the support of $\phi^0_{\rho_n}$ and is supported on, say, $B_3$, that
\begin{equ}
(\cP^n_t f-f\big)^{[0],n}  = \Big(\sum_{k=-2}^{2} \eta(k) (\tilde \mu_{t,n}(k)-1)\delta e_k \Big) *_n (f^{[0],n})
\end{equ}
and 
\begin{equ}
\big\|\sum_{k=-2}^{2} \eta(k) (\tilde \mu_{t,n}(k)-1)\delta e_k\big\|_{L^1(\Pi_n)}
\end{equ}
follows from $|\tilde\mu_{t,n}(w)-1|\lesssim t|w|^2 e^{-\kappa t|w|^2}\lesssim t^{\delta/2}$ for any $k\in\{-2,\ldots,2\}$ and $\delta\in[0,2]$
\end{proof}

\begin{lemma}\label{lem:discrete-Schauder}
Let $f,g:\Lambda_n\times\Pi_n\to\R$ satisfy for all $(t,x)\in\Lambda_n\times\Pi_n$
\begin{equ}
f_t(x)=\int_0^t\cP^n_{\kappa_n(t-s)}g_{\kappa_n(s)}(x)\,ds.
\end{equ}
Let further $\alpha\in\R$, $p,q\in[1,\infty]$, and $\delta>0$. Then there exists a constant $N=N(c)$  such that for all
 $t\in\Lambda_n\cap[0,1]$ one has the bound
\begin{equ}\label{eq:discrete-Schauder}
\|f_t\|_{B^{\alpha+2}_{p,q}(\Pi_n)}\leq \sup_{s\in\Lambda_n\cap[0,t)}\|g_s\|_{B^{\alpha}_{p,q}(\Pi_n)}.
\end{equ}
\end{lemma}
\begin{proof}
The case $t=0$ is trivial.
For $t=kh$, $k\geq 1$, we follow the classical argument, see e.g. \cite[Lem.~A.9]{GIP}, and split the integral at an intermediate time $\ell h$, $\ell<k$.
One writes
\begin{equs}
f_t^{[j],n}
&=\int_0^{t}\cP_{\kappa_n(t-s)}^ng_{\kappa_n(s)}^{[j],n}\,ds
\\
&=h\sum_{m=0}^{k-1}\cP^n_{(k-1-m)h}g^{[j],n}_{mh}
\\
&=
h\sum_{m=0}^{\ell-1}\cP^n_{(k-1-m)h}g^{[j],n}_{mh}
+h\sum_{m=\ell}^{k-1}\cP^n_{(k-1-m)h}g^{[j],n}_{mh},
\end{equs}
and estimates the two terms as follows.

For the first, fix any $\eps>0$. Using \eqref{usefull estimate} from the proof of Lemma~\ref{lem:discrete-semigroup1} with $\delta=1+\eps$ one finds
\begin{equs}
\big\|h\sum_{m=0}^{\ell-1}\cP^n_{(k-1-m)h}g^{[j],n}_{mh}\big\|_{L^p(\Pi_n)}&\leq h\sum_{m=0}^{\ell-1} \big\|\cP^n_{(k-1-m)h}g^{[j],n}_{mh}\big\|_{L^p(\Pi_n)} \\
&\leq h\sum_{m=0}^{\ell-1}  ((k-1-m)h)^{-(1+\eps)}2^{-2j(1+\eps)} \| g^{[j],n}_{mh}\|_{{L^p(\Pi_n)}}\\
&\leq 2^{-2j(1+\eps)}h^{-\eps}\sum_{m=0}^{\ell-1}  (k-1-m)^{-(1+\eps)}   \| g^{[j],n}_{mh}\|_{{L^p(\Pi_n)}}\\
&\leq 2^{-2j(1+\eps)}h^{-\eps} \sum_{m=k-\ell}^{k-1} \frac{1}{m^{1+\eps}}   \| g^{[j],n}_{(k-1-m)h}\|_{{L^p(\Pi_n)}} \ .
\end{equs}
For the second term one simply writes
\begin{equ}
\|h\sum_{m=\ell}^{k-1}\cP^n_{(k-1-m)h}g^{[j],n}_{mh}\|_{{L^p(\Pi_n)}} \leq h\sum_{m=\ell}^{k-1} \| g^{[j],n}_{mh}\|_{{L^p(\Pi_n)}}. 
\end{equ}
Therefore, for any function $j\mapsto \ell(j)$ the two bounds above give
\begin{equs}
&\| 2^{(\alpha+2)j}\|f_t^{[j],n}\|_{L^p{(\Pi_n})} \|_{\ell^q} \\
&\leq\Big\|  2^{-2j\eps}h^{-\eps}\sum_{m=k-\ell(j)}^{k-1} \frac{1}{m^{1+\eps}}  2^{\alpha j} \| g^{[j],n}_{(k-1-m)h}\|_{{L^p(\Pi_n)}} + h2^{2j}\sum_{m=\ell(j)}^{k-1} 2^{\alpha j}\| g^{[j],n}_{mh}\|_{{L^p(\Pi_n)}} \Big\|_{\ell^q} \\
& \leq\Big\|2^{-2j\eps}h^{-\eps}\sum_{m=k-\ell(j)}^{k-1} \frac{1}{m^{1+\eps}}\Big\|_{\ell^\infty} \| 2^{\alpha j} \| g^{[j],n}_{(k-1-m)h}\|_{{L^p(\Pi_n)}} \|_{\ell^q}  
\\
&\qquad\qquad+\Big\| h2^{2j}\sum_{m=\ell(j)}^{k-1} 1\Big\|_{\ell^\infty} \| 2^{\alpha j}\| g^{[j],n}_{mh}\|_{{L^p(\Pi_n)}} \|_{\ell^q} \\
&\lesssim  \Big( \mathbf{1}_{\ell(j)>0} \Big\|\frac{2^{-2\eps j}}{\big(h(k-\ell(j))\big)^{\eps}}\Big\|_{\ell^\infty}+  \Big\| h2^{2j}(k-\ell(j)) \Big\|_{\ell^\infty}\Big) \sup_{s\in\Lambda_n\cap[0,t)}\|g_s\|_{B^{\alpha}_{p,q}(\Pi_n)}\ .
\end{equs}

Making the choice $\ell(j)= (k-\lceil h^{-1}2^{-2j}\rceil) \vee 0$, the quantity in the big parentheses is of order $1$. Indeed, if $\ell(j)=0$ the first term in the parenthesis does not contribute and in this case $k \leq \lceil h^{-1}2^{-2j}\rceil$ which implies that the second term is bounded, and if $\ell(j)= (k-\lceil h^{-1}2^{-2j}\rceil)$ both terms are clearly bounded as well.
\end{proof}

\subsection{Heat kernel error bounds}
We recall $3$ error bounds for heat kernels from \cite{Mate}.
The following is \cite[Lem.~2.2.9]{Mate}.
\begin{lemma}     \label{lem:det-rate}
For any  $\alpha \in (0, 1)$ there exists a constant $N=N(\alpha,c)$ such that for all $\psi \in C^\alpha(\T)$, $t \in [0,1]$, and  $y \in \T$, we have 
\begin{equs}
| \cP^n_t \psi(y) - \cP_t \psi (y) | \leq N n^{-\alpha} \| \psi \|_{C^\alpha(\T)}.    \label{eq:rate-deterministic}
\end{equs}
\end{lemma}
The following is \cite[Lem.~2.2.7]{Mate} (see also\cite[Lemma~3.3]{Gy0}).
\begin{lemma}   \label{lem:Pn-P}
Let $\beta \in [0, 2]$. Then there exists a constant $N(\beta , c)$ such that 
for all $t \in  [h, 1]$, $x \in \T$ one has the bound
\begin{equs}\label{eq:Pn-P}
\big\| p_t(x,\cdot)-p^{n}_{\kappa_n(t)}(x,\rho_n(\cdot)) \big\|^2_{L^2(\T)} \leq  N n^{-\beta} t^{-(\beta+1)/2}.
\end{equs} 
\end{lemma}
The following is a crucial tool from \cite[Lem.~3.3.1]{Mate} that we alluded to in Remark \ref{rem:BDG}.
Note that pulling the norm inside the integral in \eqref{eq:names} gives only a rate $1/2$, therefore using the regularisation of the noise is essential (and is done by stochastic sewing \cite{Khoa} in \cite{Mate}).
We remark that technically \cite[Lem.~3.3.1]{Mate} is stated with $w$ below defined by $\cP^n$ and $O^n$ in place of $\cP$ and $O$; the choice we make here would only make the proof easier.
Set
\begin{equs}\label{eq:w}
w_t(x)= \cP_t  \psi(x) + O_t(x).
\end{equs}
\begin{lemma}\label{lem:integral}
Let $p\geq 2$ and $\eps\in(0,1/4)$, and suppose that for some constant $K$ one has $\|\psi\|_{C^{1/2}(\T)}\leq K$.
Then there exists a constant  $N=N(p, \eps,c, K)$  such that  for all bounded measurable function $g:\R\to\R$ and all $n\in\N$, one has the bound
\begin{equs}
\sup_{t\in[0,1]}\sup_{x\in \T}\Big\|\int_0^t\int_T & p_{t-r}(x,y)\Big( g\big(w_r(y)\big)-g\big(w_{\kappa_n(r)}(\rho_n(y)\big)\Big)\,dy\,dr\Big\|_{L^p(\Omega)}
\\
&\leq N\|g\|_{L^\infty(\R)}n^{-1+\eps}.
\label{eq:names}
\end{equs}
\end{lemma}

\section{Main proofs}
\subsection{Bounds on the linear solutions in Besov norms}

We start with a very specific case of Kolmogorov's theorem that is sufficient for our setting.
\begin{prop}\label{prop:kolmogorov}
Let $\{X_k\}_{k\in \mathbb{Z}}$ be a family of $\C$-valued Gaussian random variables that are mutually complex orthogonal (i.e. $\E X_k\overline{X_\ell}=0$ for $k\neq \ell$). Denote $a_k=\sqrt{\E|X_k|^2}$.
Then the series
$$F:= \sum_{k\in\Z} X_k e_k $$ 
defines a random distribution, which for any $\alpha\in\R$ and $q\in[1,\infty)$ satisfies the bounds
\begin{equ}\label{besov estimate}
\E\| F\|_{B^\alpha_{\infty,\infty}(\T)}\leq\E\| F\|_{B^\alpha_{\infty,q}(\T)}\leq N \Big( \sum_{j\in\N} 2^j \big( \sum_{k\in \Z}(2^{\alpha j}\phi^j_1(k)a_k)^2\big)^{\frac{q}{2}}\Big)^{1/q},
\end{equ}
where $N=N(q)$.

Similarly,
the random function
$$F= \sum_{k=-n}^{n-1} X_k \delta e_k $$ 
satisfies for any $\alpha\in\R$ and $q\in[1,\infty)$ the bounds
\begin{equ}\label{besov estimate-n}
\E\| F\|_{B^\alpha_{\infty,\infty}(\Pi_n)}\leq\E\| F\|_{B^\alpha_{\infty,q}(\Pi_n)}\leq N \Big( \sum_{j\in\N} 2^j \big( \sum_{k=-n}^{n-1}(2^{\alpha j}\phi^j_{\rho_n}(k)a_k)^2\big)^{\frac{q}{2}}\Big)^{1/q},
\end{equ}
where $N=N(q)$.
\end{prop}

\begin{proof}
The proof of the two parts are identical, so we only give the first.
The first inequality in \eqref{besov estimate} is trivial. For any $j\in\N$, $F^{[j]}$ is a smooth function, and one can write
\begin{align*}
\E|F^{[j]}|^2(x)  \leq \sum_{k, k'\in\Z}\phi^j_1(k) \phi^j_1(k') \big(\E X_k\overline{X_{k'}}\big) e_k(x) \bar{e}_{k'}(x) = \sum_{k\in \Z}|\phi^j_1(k)a_k|^2.
\end{align*}
By the equivalence of moments of Gaussian random variables we then have
\begin{equ}\label{eq:33}
\E\|F^{[j]}\|_{L^q(\T)}^q= \int_\T \E|F^{[j]}(x)|^q\, dx \lesssim \int_\T \big(\E|F^{[j]}(x)|^2\big)^{\frac{q}{2}} \,dx = \big(\sum_{k\in \Z}(\phi^j_1(k)a_k)^2\big)^{\frac{q}{2}}.
\end{equ}
Therefore by Bernstein's inequality
\begin{align*}
\E\| F\|_{B^\alpha_{\infty,q}(\T)}^q = \E\sum_{j\in\N} (2^{\alpha j}\|F^{[j]}\|_{L^\infty(\T)})^q\lesssim \E\sum_{j\in\N} 2^j (2^{\alpha j}\|F^{[j]}\|_{L^q(\T)})^q,
\end{align*}
and by \eqref{eq:33} we get the second inequality in \eqref{besov estimate}.
\end{proof}
\begin{remark}
The proof only uses the equivalence of moments property of Gaussians, and so the statement immediately extends to any other family of probability distributions with the same property (e.g. random variables from a fixed Wiener chaos).
\end{remark}
We first show the bounds of desired order for the linear solutions. Similar in spirit bounds can be found in \cite[Lem.~3.4.]{MZ}, with a number of differences: therein the $2$-dimensional case is considered, without discretisation in time, and with Galerkin approximation in space (and of course due to the difference in dimension the regularities are shifted by $1/2$ compared to the ones below).
\begin{lemma}\label{lem:upperbound}
For any $n\in\N$, $t\in[0,1]$, and $q\in[1,\infty)$ one has the bounds
\begin{equ}\label{eq:mainObound-cont}
\mathbb{E}\|O_t-\iota O^n_t\|_{B^\alpha_{\infty,\infty}(\T)}^q\leq 
\begin{cases}
N n^{-q(1/2-\alpha-1/q)}&\quad \text{if }\alpha\in (-1/2,1/2),
\\
N n^{-q} & \quad\text{if }\alpha<-1/2,
\end{cases}
\end{equ}
\begin{equ}\label{eq:mainObound-disc}
\mathbb{E}\|\delta O_t-O^n_t\|_{B^\alpha_{\infty,\infty}(\Pi_n)}^q\leq 
\begin{cases}
N n^{-q(1/2-\alpha-1/q)}&\quad \text{if }\alpha\in (-1/2,1/2),
\\
N n^{-q} & \quad\text{if }\alpha<-1/2,
\end{cases}
\end{equ}
where $N=N(\alpha,q)$.
\end{lemma}
\begin{proof}
By Jensen's inequality, it clearly suffices to bound for large enough $q$.
We start with the bound \eqref{eq:mainObound-cont}.
We wish to apply Proposition \ref{prop:kolmogorov} with $F=O- \iota O^n$.
To verify the complex orthogonality of the coordinates, take $k\neq \ell$ and first note that from \eqref{eq:O Fourier} (\eqref{eq:On Fourier}, resp.), It\^o's isometry, and the orthogonality of the $e_k$-s ($\delta e_k$-s, resp.) it is clear that for $\ell\neq k$ we have $\E\langle O_t, e_\ell\rangle\overline{\langle O_t, e_k\rangle}=0$, $\E\langle \iota O^n_t, e_\ell\rangle\overline{\langle \iota O^n_t, e_k\rangle}=0$, resp.
Furthermore, elementary calculation shows
\begin{equs}
\E\langle \iota O^n_t, e_\ell\overline{\rangle \langle O_t, e_k\rangle}&=\int_0^t\gamma_{s,t}\int_\T e_k(y)\overline{e}_\ell(\rho_n(y))\,dy\,ds
\\
&=\int_0^t\gamma_{s,t}\frac{e^{2\pi i\frac{k}{2n}}}{2\pi i k}\sum_{j=0}^{2n-1}e^{2\pi i j\frac{k}{2n}}e^{-2\pi i j\frac{\ell}{2n}}\,ds=0 
\end{equs}
with some bounded function $\gamma_{s,t}$.

To apply Proposition \ref{prop:kolmogorov}, we need to bound $a_\ell$. For $\ell\in\{-n,\ldots,n-1\}$ we have
\begin{equs}
\mathbb{E}&|\langle O_t-\iota O^n_t, e_\ell\rangle|^2 = \int_0^t\int_\T \big| e^{- 4 \pi^2 \ell^2 (t-s) }  \bar{e}_\ell(y)-(1+h\lambda_\ell^n)^{(\kappa_n (t-s)/h)} \overline{e_\ell^n(\rho_n(y))}  \big|^2\,dy\,ds\\
&\lesssim \int_0^t\int_\T \big| e^{- 4 \pi^2 \ell^2 (t-s) }  \bar{e}_\ell(y)-e^{- 4 \pi^2 \ell^2 (t-s) } \overline{e_\ell^n(\rho_n(y))}  \big|^2\,dy\,ds\\
&\quad+ \int_0^t\int_\T \big| e^{- 4 \pi^2 \ell^2 (t-s) } \overline{e_\ell^n(\rho_n(y))} -(1+h\lambda_\ell^n)^{(\kappa_n (t-s)/h)} \overline{e_\ell^n(\rho_n(y))}  \big|^2\,dy\,ds\\
&\leq \int_0^te^{- 8 \pi^2 \ell^2 (t-s) } \,ds \int_\T \big|  \bar{e}_\ell(y)-\overline{e_\ell^n(\rho_n(y))}  \big|^2\,dy\\
&\quad+ \int_0^t \big( e^{- 4 \pi^2 \ell^2 (t-s) } -(1+h\lambda_\ell^n)^{(\kappa_n (t-s)/h)}\big)^2 \,ds\int_\T  |\overline{e_\ell^n(\rho_n(y))} |^2\,dy.
\\
&\lesssim \int_0^te^{- 8 \pi^2 \ell^2 (t-s) } \,ds \int_\T \big|  \bar{e}_\ell(y)-\overline{e_\ell^n(\rho_n(y))}  \big|^2\,dy\\
&\quad+\int_0^t \big( e^{- 4 \pi^2 \ell^2 (t-s) } -e^{- \lambda^n_\ell (t-s)  } \big)^2\,ds\\
&\quad+  \int_0^t \big( e^{- \lambda^n_\ell (t-s) } -e^{- \lambda^n_\ell \kappa_n(t-s) } \big)^2 \,ds\\
&\quad+  \int_0^t   \big(e^{- \lambda^n_\ell \kappa_n(t-s) } -(1+h\lambda_\ell^n)^{(\kappa_n (t-s)/h)}\big)^2 \,ds
\\
&=:I^1+I^2+I^3+I^4.
\end{equs}
We claim that each $I^j$ satisfies
\begin{equ}\label{eq:I bounds}
I^j\lesssim n^{-2}.
\end{equ}
First one has trivially
$$ \int_\T \big|  \bar{e}_\ell(y)-\overline{e_\ell^n(\rho_n(y))}  \big|^2\lesssim \ell^2/n^2,$$
and thus by the elementary inequality $\int_0^t e^{-\alpha s} ds \leq 1/\alpha$, we get \eqref{eq:I bounds} for $I^1$.
Next,
we have
$$
I_2\lesssim \int_0^t e^{-32 \ell^2 (t-s)} \ell^4 (\gamma_\ell^n-1)^2 (t-s)^2\,ds\lesssim   \int_0^t e^{-32 \ell^2 (t-s)}\ell^8 n^{-4} (t-s)^2 \,ds$$
where we used 
$\lambda_\ell\leq\lambda_\ell^n\leq -16 \ell^2$ in the first inequality (see \eqref{eq:gamma-prop1}) and \eqref{eq:gamma-prop2}
in the second.
Now by the inequality $\int_0^t e^{-\alpha s} s^2 ds \leq 2/\alpha^3$ we conclude
$I^2\lesssim \ell^2n^{-2}$, and since $|\ell|\leq n$, this yields \eqref{eq:I bounds} for $I^2$ as claimed.
Next, we have
$$
I_3\leq
\int_0^t \big( |t-s-\kappa_n(t-s)|\lambda_\ell^n {e^{- \lambda^n_\ell \kappa_n(t-s) }}\big)^2\,ds\lesssim \int n^{-4}\ell^4 e^{-16 \ell^2 (t-s)} \,ds\lesssim \ell^2 n^{-4},$$
as desired.
Finally, arguing as in \cite{Mate}, we find
\begin{align*}
I_4&\lesssim 
\int_0^t   e^{-16\ell^2(t-s)} |\kappa_n(t-s) h^{-1}|^2 |h\lambda_\ell^n|^4\,ds 
\\
&\lesssim \int_0^t   e^{-16\ell^2(t-s)}  (t-s)^2 n^{-4} \ell^{8} \,ds
\lesssim \ell^2 n^{-4}.
\end{align*}
yielding again \eqref{eq:I bounds} as desired. Therefore the proof of \eqref{eq:I bounds} is complete.

Moving on to $\ell\in\Z\setminus\{-n,\ldots,n-1\}$, it is easy to see that
\begin{equ}
\mathbb{E}|\langle O_t-\iota O^n_t, e_\ell \rangle|^2=\mathbb{E}|\langle O_t, e_\ell \rangle|^2=\int_0^te^{-8\pi^2\ell^2(t-s)}\,ds\lesssim \ell^{-2}.
\end{equ}
Altogether we get
\begin{equ}
|a_\ell |^2=\mathbb{E}|\langle O_t-\iota O^n_t, e_\ell\rangle|^2\lesssim n^{-2}\wedge\ell^{-2}.
\end{equ}
By Proposition \ref{prop:kolmogorov} we therefore have 
\begin{align*}
\mathbb{E}\|O_t-\iota O^n_t\|_{B^\alpha_{\infty,\infty}(\T)}^q &\lesssim  \sum_{j\in\N} 2^j \Big( \sum_{k\in \Z}\big(2^{\alpha j}\phi^j_1(k)(n^{-1}\wedge k^{-1})\big)^2\Big)^{\frac{q}{2}}.
\end{align*}
Summing first over $j$ such that $2^j\leq n$, we get a bound of order
\begin{equ}
n^{-q}\sum_{2^j\leq n}2^{j(1+q/2+\alpha q)}
\end{equ}
If $\alpha>-1/2$, then this is of order $n^{1-q(1/2-\alpha)}$, while for $\alpha<-1/2$, for large enough $q$, it is of order $n^{-q}$.
Both of these bounds are of the required order.
On the complement regime $2^j>n$, we get a bound of order
\begin{equ}
\sum_{2^j>n}2^{j(1-q/2+\alpha q)}.
\end{equ}
For any $\alpha<1/2$, for large enough $q$, this is of order $n^{1-q(1/2-\alpha)}$. This is again of the required order.

The proof of \eqref{eq:mainObound-disc} is very similar.
The orthogonality of the coordinates follows the same way.
To estimate 
$\mathbb{E}|\langle \delta O_t- O^n_t, \delta e_\ell\rangle_n|^2 $, note that for all $\ell\in \{-n,\ldots,n-1\}$ and all $j\in \mathbb{Z}$ one has
$\delta e_{\ell+ j2n}= \delta e_\ell$, and thus 
$$
\langle \delta O_t, \delta e_\ell \rangle_n = \sum_{j\in\Z} \langle O_t, e_{\ell+ j2n}\rangle= \langle O_t, e_{\ell}\rangle +\sum_{j\neq 0} \langle O_t, e_{\ell+ j2n}\rangle\ .$$
Therefore $\langle \delta O_t- O^n_t, \delta e_\ell\rangle_n= \langle O_t- \iota O^n_t, e_{\ell}\rangle +\sum_{j\neq 0} \langle O_t, e_{\ell+ j2n}\rangle$, which implies
\begin{align*}
|a_\ell|^2= \E |\langle \delta O_t-  O^n_t,\delta e_\ell\rangle_n |^2 &\lesssim  \E|\langle O_t- \iota O^n_t, e_{\ell}\rangle|^2 + \E| \sum_{j\neq 0} \langle O_t, e_{\ell+ j2n}\rangle|^2\\
&\lesssim  \frac{1}{n^2} + \sum_{{j\neq 0}} \frac{1}{ (\ell+ j2n)^2}\\
&\lesssim  \frac{1}{n^2} \ .
\end{align*}
This is precisely the same bound as before, so by applying the second part of Proposition \ref{prop:kolmogorov} we get \eqref{eq:mainObound-disc} as before.
\end{proof}

Complementing the upper bounds in Lemma\ref{lem:upperbound} is the lower bound in Proposition~\ref{prop:lowerbound} that we prove below.

\begin{proof}[Proof of Proposition \ref{prop:lowerbound}]
Note that in the setting of the statement, we have $u=O$, $u^n=O^n$.
As seen in the last proof, $\langle \delta O_t- O^n_t, \delta e_1\rangle_n= \langle O_t- \iota O^n_t, e_{1}\rangle +\sum_{j\neq 0} \langle O_t, e_{1+ j2n}\rangle$, and since the terms are Gaussian and independent, it suffices to bound the variance of $\langle O_t- \iota O^n_t, e_{1}\rangle$ from below by a positive constant times $n^{-2}$.
Recall that 
\begin{equ}
\mathbb{E}|\langle O_t-\iota O^n_t, e_1\rangle|^2 = \int_0^t\int_\T \big| e^{- 4 \pi^2  (t-s) }  \bar{e}_1(y)-(1+h\lambda_1^n)^{(\kappa_n (t-s)/h)} \overline{e_1^n(\rho_n(y))}  \big|^2\,dy\,ds.
\end{equ}
Recall that for any two unit vectors $z_1,z_2\in\C$ and any two constants $c_1,c_2>0$, one has $|c_1 z_1-c_2z_2|\geq (c_1\wedge c_2)|z_1-z_2|$. Therefore 
\begin{equ}
\mathbb{E}|\langle O_t-\iota O^n_t, e_1\rangle|^2\gtrsim \int_\T \big|   \bar{e}_1(y)- \overline{e_1^n(\rho_n(y))}  \big|^2\,dy\gtrsim n^{-2},
\end{equ}
as claimed.
\end{proof}

\subsection{A priori bounds on the approximations}\label{sec:apriori}
The purpose of this section is to derive an a priori bound on the approximation $u^n$.
Let $v^n=u^n-O^n$,
which solves the recursion
\begin{align}
v^n_{{t}+h}( {x}) &=v^n_{{t}}( {x})+h \Delta_n v^n_{ t}( x)
+h F\big(O^n_t(x)+v^n_{ t}({x})\big), \label{eq:approx v classical form}
\end{align}
$(t,x)\in\Lambda_n\times\Pi_n$,
with initial data $v^n_0(x)=\psi^n(x)$, $x\in\Pi_n$.
Alternatively, in mild form we have
\begin{equ}
v^n_{ t}( x)
=\mathcal{P}^{n}_{ t} \psi( x)+\int_0^{ t} \mathcal{P}^{n}_{\kappa_n( t-s)} F \big(O^n_{\kappa_n(s)}+v^n_{\kappa_n(s)}\big)( x)\,ds.
 \label{eq:approx v mild form}
\end{equ}
Throughout the section we fix an even integer $\mu>\nu$, and aim to bound the $L^\mu(\Pi_n)$-norm of $v^n$.
Introduce the quantities
\begin{equ}
R^n=1+\max_{t\in\Lambda_n\cap[0,1]}\|O^n_t\|_{L^\infty(\Pi_n)},\qquad A^n_t=\max_{s\in\Lambda_n\cap[0,t]}\|v^n_s\|_{L^{\mu}(\Pi_n)}.
\end{equ}
\begin{theorem}\label{thm:apriori}
Fix a $K$ such that $\|\delta\psi\|_{L^\mu(\Pi_n)}\leq K$ for all $n\in\N$.
Then there exists a constant $N=N(c,c_0,\ldots,c_\nu,\nu,\mu,K)$ such that
on the event
\begin{equ}
\Omega_n=\{(R^n)^{\nu(\nu+\mu)}\leq n^{1-(\nu-1)/\mu}\}
\end{equ}
one has the bound
\begin{equ}\label{eq:apriori-main-mu}
A^n_1\leq N(R^n)^{(\nu+\mu-1)/\mu}.
\end{equ}
\end{theorem}
\begin{remark}\label{rem:trunc}
To get a more direct analogy to the a priori bound of \cite{Jentzen, Wang}, note that Theorem \ref{thm:apriori} can easily be turned into an $L^p(\Omega)$ bound on a truncated scheme.
Take $N$ from Theorem \ref{thm:apriori} and define
\begin{equ}
\tau_n:=\inf\{t\in\Lambda_n: \|u^n_t\|_{L^\mu(\Pi_n)}\geq (N+1) n\}
\end{equ}
and let the truncated scheme be $\tilde u^n_t=u^n_{t\wedge\tau_n}$.
Then, on $\Omega_n$ the bound \eqref{eq:apriori-main-mu} implies that $A_1^n\leq N n$, and therefore $\tau_n\geq 1$.
Notice furthermore that as a consequence of Gaussianity, $\mathbb{P}(\Omega_n^c)$ decays faster than any power of $n$.
Therefore
\begin{equs}
\big\|&\max_{t\in\Lambda_n\cap[0,1]}\|\tilde u^n_t\|_{L^\mu(\Pi_n)}\big\|_{L^p(\Omega)}
\\
&\leq\big\|\mathbf{1}_{\Omega_n}\max_{t\in\Lambda_n\cap[0,1]}\|u^n_t\|_{L^\mu(\Pi_n)}\big\|_{L^p(\Omega)}
+\big\|\mathbf{1}_{\Omega_n^c}\max_{t\in\Lambda_n\cap[0,1]}\|\tilde u^n_t\|_{L^\mu(\Pi_n)}\big\|_{L^p(\Omega)}
\\
&\lesssim \big\|(R^n)^{(\nu+\mu-1)/\mu}\|_{L^p(\Omega)}
+n\mathbb{P}(\Omega_n^c)\lesssim 1.
\end{equs}
\end{remark}

Before the proof of Theorem \ref{thm:apriori} we start by outlining the argument, to motivate the intermediate Lemma \ref{lem:apriori-intermediate} below.
Let $t\in\Lambda_n\cap[0,1]$ and write $k=t/h$.
Then one has
\begin{equ}
\|v^n_t\|_{L^{\mu}(\Pi_n)}^{\mu}=\|\delta\psi\|_{L^{\mu}(\Pi_n)}^{\mu}+\sum_{j=0}^{k-1}\|v_{(j+1)h}^n\|_{L^{\mu}(\Pi_n)}^{\mu}-\|v_{jh}^n\|_{L^{\mu}(\Pi_n)}^{\mu}. 
\end{equ}
One can then rewrite the summands by the binomial theorem
\begin{equs}
\|v^n_t\|_{L^{\mu}(\Pi_n)}^{\mu}&=\|\delta\psi\|_{L^{\mu}(\Pi_n)}^{\mu}+\sum_{j=0}^{k-1}\mu\bigscal{v^n_{(j+1)h}-v^n_{jh},(v^n_{jh})^{\mu-1}}_n
\\
&\qquad\qquad+\sum_{j=0}^{k-1}\sum_{\ell=2}^{\mu}\binom{\mu}{\ell}\bigscal{(v^n_{(j+1)h}-v^n_{jh})^\ell,(v^n_{jh})^{\mu-\ell}}_n. \label{eq:apriori-sum}
\end{equs}
In the first sum we will leverage the monotonicity of the nonlinearity. In the second sum we use the fact that the time increments appear with power at least $2$, therefore a bound on this increment with a power of $h$ larger than $1/2$ makes the total sum `small'. This is the content of Lemma \ref{lem:apriori-intermediate}, which does use the sign of the nonlinearity.

\begin{lemma}\label{lem:apriori-intermediate}
For any $\eps>0$ there exists a constant $N=N(\eps,c,c_0,\ldots,c_\nu,\nu,\mu,K)$ such that for all $t\in\Lambda_n\cap[0,1]$ the following bound holds:
\begin{equ}\label{eq:crude-apriori-main}
\|v^n_{t+h}-v^n_t\|_{L^\mu(\Pi_n)}\leq N h^{1-(\nu-1)/(2\mu)-\eps}\big((R^n)^\nu+(A^n_{t})^\nu\big).
\end{equ}
\end{lemma}
\begin{proof}
Recall that $
v^n_{t+h}-v^n_t=\cP^n_hv^n_t+hF(u^n_t)-v^n_t,$
and applying \eqref{eq:reverse-discrete-Lp} we immediately get
\begin{equs}
h\|F(u^n_t)\|_{L^\mu(\Pi_n)}&\lesssim h n^{(\nu-1)/\mu} \|F(O^n_t+v^n_t)\|_{L^{\mu/\nu}(\Pi_n)}
\\&
\lesssim h n^{(\nu-1)/\mu}\big((R^n)^\nu+(A^n_t)^\nu\big).\label{eq:crude-aprio1}
\end{equs}
It remains to bound $\cP^n_h v^n_t-v^n_t$. By Lemma \ref{lem:discrete-Schauder} and Proposition \ref{prop:equivalence} we have
\begin{equs}
\|v_t\|_{B^{2}_{\mu/\nu,\infty}(\Pi_n)}&\lesssim \sup_{s\in\Lambda_n\cap[0,t)}\|F(O^n_s+v^n_s)\|_{B^{0}_{\mu/\nu,\infty}(\Pi_n)}
\\&\lesssim\sup_{s\in\Lambda_n\cap[0,t)}\|F(O^n_s+v^n_s)\|_{L^{\mu/\nu}(\Pi_n)}
\\&\lesssim (R^n)^\nu+(A^n_{t})^\nu.
\end{equs}
Therefore by Lemma \ref{lem:discrete-Besov-embedding} we have
\begin{equ}
\|v_t\|_{B^{2-(\nu-1)/\mu}_{\mu,\infty}(\Pi_n)}\lesssim \|v_t\|_{B^{2}_{\mu/\nu,\infty}(\Pi_n)}\lesssim (R^n)^\nu+(A^n_{t})^\nu.
\end{equ}
Finally, using Proposition \ref{prop:equivalence} and Lemma \ref{lem:discrete-semigroup2}, we have
\begin{equs}
\|\cP^n_h v^n_t-v^n_t\|_{L^\mu(\Pi_n)}&\lesssim\|\cP^n_h v^n_t-v^n_t\|_{B^\eps_{\mu,\infty}(\Pi_n)}
\\
&\lesssim h^{1-(\nu-1)/(2\mu)-\eps/2}\|v_t\|_{B^{2-(\nu-1)/\mu}_{\mu,\infty}(\Pi_n)}
\\
&\lesssim h^{1-(\nu-1)/(2\mu)-\eps/2}\big((R^n)^\nu+(A^n_{t})^\nu\big).\label{eq:crude-aprio2}
\end{equs}
The bounds \eqref{eq:crude-aprio1} and \eqref{eq:crude-aprio2} yield \eqref{eq:crude-apriori-main} with $\eps/2$ in place of $\eps$.
\end{proof}


\begin{proof}[Proof of Theorem \ref{thm:apriori}]
We start from the decomposition \eqref{eq:apriori-sum} and bound the terms one by one. Recall that $v^n_{(j+1)h}-v^n_{jh}=\Delta_n v^n_{jh}+h F(O^n_{jh}+v^n_{jh})$.
At this point the sign of the nonlinearity comes into play: fixing an arbitrary parameter $a\in\R$, the leading order term in the polynomial $y\mapsto F(a+y)y^{\mu-1}$ is $c_\nu  y^{\nu+\mu-1}$.
It follows that there exists a constant $N_0\lesssim 1$ such that  $F(a+y)y^{\mu-1}$ is negative for $|y|\geq N_0 |a|$.
On the other hand, if $|y|\leq N_0|a|$, then we can bound $F(a+y)y^{\mu-1}$ trivially by a constant times $1+|a|^{\nu+\mu-1}$.
Therefore we get
\begin{equ}
\bigscal{F(O^n_{jh}+v^n_{jh}),(v_{jh}^n)^{\mu-1}}_n\lesssim 1+\|O^n_{jh}\|^{\nu+\mu-1}_{L^{\nu+\mu-1}(\Pi_n)}\leq (R^n)^{\nu+\mu-1}.\label{eq:proof-apriori1}
\end{equ}
As for the contribution of the term with the Laplacian, we can argue similarly to \cite[Lem.~3.1]{GG}: first notice that with $\delta_n^\pm:C(\Pi_n)\to C(\Pi_n)$ defined by $\delta_n^{\pm}f(x)=\pm(2n)^{-1}\big(f(x)-f(x\pm (2n)^{-1})\big)$, we have $\Delta_n=\delta_n^+\delta_n^-$. Further, $\delta_n^+$ is the adjoint of $-\delta_n^-$ with respect to the inner product $\scal{\cdot,\cdot}_n$.
Finally, for any $f\in C(\Pi_n)$, and any $p\geq 2$ one has
\begin{equs}
\delta_n^-(|f|^{p-2}f)(x)&=\delta_n^-f(x)\int_0^1 (p-1)\big| (1-\theta)f(x)+\theta f(x-(2n)^{-1})\big|^{p-2}\,d\theta
\\&=: \delta_n^-f(x) F_{p,n}f(x),
\end{equs}
with the mapping $F_{p,n}:C(\Pi_n)\to C(\Pi_n)$ defined by the above equality clearly satisfying $F_{p,n}f\geq 0$ for any $f$.
Putting these observations together,
\begin{equ}
\bigscal{\Delta_n v^n_{jh},(v^n_{jh})^{\mu-1}}_n=-\bigscal{\delta_n^- v^n_{jh},\delta_n^-\big((v^n_{jh})^{\mu-1}\big)}_n=-\bigscal{\delta_n^- v^n_{jh},(F_{\mu,n}v^n_{jh})\delta_n^- v^n_{jh}}_n\leq 0.\label{eq:proof-apriori2}
\end{equ}
Moving on to the terms of the last type in \eqref{eq:apriori-sum}, by H\"older's inequality and Lemma \ref{lem:apriori-intermediate} (with $\eps>0$ to be chosen momentarily) we can write for any $j=0,\ldots,k-1$ and $\ell=2,\ldots \mu-1$,
\begin{equs}
\big|\bigscal{(v^n_{(j+1)h}-v^n_{jh})^\ell,(v^n_{jh})^{\mu-\ell}}_n\big|
&\leq \|v^n_{(j+1)h}-v^n_{jh}\|_{L^\mu(\Pi_n)}^\ell\|v^n_{jh}\|_{L^\mu(\Pi_n)}^{\mu-\ell}
\\&\lesssim h^{\ell(1-(\nu-1)/(2\mu)-\eps)}\big((R^n)^\nu+(A^n_{jh})^\nu\big)^\ell(A^n_{jh})^{\mu-\ell}
\\
&\lesssim h^{2(1-(\nu-1)/(2\mu)-\eps)}\big((R^n)^{\nu\mu}+(A^n_{jh})^{\nu\mu}\big)\label{eq:proof-apriori3}.
\end{equs}
Substituting \eqref{eq:proof-apriori1}-\eqref{eq:proof-apriori2}-\eqref{eq:proof-apriori3} into \eqref{eq:apriori-sum}, summing up, recalling the assumption on $\|\delta\psi\|_{L^\mu(\Pi_n)}$, and keeping in mind that $k\leq h^{-1}$, we get
\begin{equ}
\|v^n_t\|^\mu_{L^\mu(\Pi_n)}\lesssim (R^n)^{\nu+\mu-1}+h^{1-(\nu-1)/\mu-2\eps}\big((R^n)^{\nu\mu}+(A^n_{t-h})^{\nu\mu}\big).
\end{equ}
Denote $\kappa=1-(\nu-1)/\mu>0$ and fix $\eps=\kappa/2$. Note that the definition of $\Omega_n$ ensures that $h^{\kappa/2}(R^n)^{\nu\mu}\leq 1$, so on the event $\Omega_n$ the above bound further simplifies to
\begin{equ}\label{eq:apriori-almostdone}
\|v^n_t\|^\mu_{L^\mu(\Pi_n)}\lesssim (R^n)^{\nu+\mu-1}+h^{\kappa/2}(A^n_{t-h})^{\nu\mu}.
\end{equ}
We now claim that this implies the claimed bound \eqref{eq:apriori-main-mu} if we choose $N$ large enough.
Clearly it suffices to prove for $n$ large enough.
We proceed by induction on $k=t/h$.
For $k=0$ we simply need to choose $N\geq K$. In the inductive step, if $A_t^n=A_{t-h}^n$, then we are clearly done, otherwise $A^n_t=\|v^n_t\|_{L^\mu(\Pi_n)}$. Denoting the implicit constant in \eqref{eq:apriori-almostdone} by $N_1$, we then have
\begin{equ}
(A_t^n)^\mu\leq N_1 (R^n)^{\nu+\mu-1}+N_1 N h^\kappa (R^n)^{\nu(\nu+\mu-1)}. 
\end{equ}
To handle the first term, we simply choose $N\geq 2N_1$. For the second, note that on $\Omega_n$ we have $h^\kappa (R^n)^{\nu(\nu+\mu-1)}\to 0$ as $n\to 0$. As a consequence, for large enough $n$ the second term can be bounded by $N/2$. This finishes the proof.
\end{proof}

\subsection{Proof of Theorem \ref{thm:main}}
In addition to our usual notation $\lesssim$, in this proof we use the notation $a\preceq b$ to denote the existence of an almost surely finite random variable $\eta$ depending only on the parameters of the problem (that is, $\psi$, $\theta$, $\eps$, and $F$) such that $a\leq \eta b$ (either on the whole $\Omega$ or on a specified subset of it). 

We write $v^n=u^n-O^n$ as before, as well as $v=u-O$.
By Lemma \ref{lem:upperbound} and Proposition \ref{prop:equivalence} it is sufficient to show
\begin{equ}
\sup_{t\in\Lambda_n\cap[0,1]} \|\delta v_{\kappa_n(t)}-v^n_{\kappa_n(t)}\|_{L^\infty(\Pi_n)}\preceq n^{-1/2+\theta+\eps}.
\end{equ}
For $(t,x)\in\Lambda_n\times\Pi_n$ we can write
\begin{equs}
\delta v_t(x)-v^n_t(x)&=\cP_t \psi(x)-\cP^n_t\psi(x)+\int_0^t\cP_{t-s}F(u_s)(x)-\cP^n_{\kappa_n(t-s)}F(u^n_{\kappa_n(s)})(x)\,ds
\\
&=\cP_t \psi(x)-\cP^n_t\psi(x)
\\
&\qquad+\int_0^t\int_{\T} p_{t-s}(x,y)\Big(F\big(u_s(y)\big)-F\big(u_{\kappa_n(s)}(\rho_n(y))\big)\Big)\,dy\,ds
\\
&\qquad+\int_0^t\int_{\T} \Big(p_{t-s}(x,y)-p^n_{\kappa_n(t-s)}(x,\rho_n(y))\Big)F\big(u_{\kappa_n(s)}(\rho_n(y))\big)\,dy\,ds
\\
&\qquad+\int_0^t\cP^n_{\kappa_n(t-s)}\Big(F(u_{\kappa_n(s)})-F(u^n_{\kappa_n(s)})\Big)(x)\,ds
\\
&=:\cE^{n,0}_t(x)+\cE^{n,1}_t(x)+\cE^{n,2}_t(x)+\cE^{n,3}_t(x).\label{eq:CE-decomp}
\end{equs}
We can first write
\begin{equ}
|\cE^{n,0}_t(x)| 
\lesssim n^{-1+\eps},\label{eq:bound-cE0}
\end{equ}
using Lemma \ref{lem:det-rate} (with $\alpha=1-\eps$) and $\psi\in C^{1-\eps}(\T)$.

Next, we have
\begin{equs}
|\cE^{n,2}_t(x)|&\leq\int_0^t\big\|p_{t-s}(x,\cdot)-p^n_{\kappa_n(t-s)}(x,\rho_n(\cdot))\big\|_{L^1(\T)}\big\|F(u_{\kappa_n(s)})\big\|_{L^\infty(\Pi_n)}\,ds
\\
&\lesssim \big(1+ \|u\|_{L^\infty([0,1]\times\T)}^\nu\big)\int_0^t\big\|p_{t-s}(x,\cdot)-p^n_{\kappa_n(t-s)}(x,\rho_n(\cdot))\big\|_{L^1(\T)}\,ds
\\
&\lesssim \big(1+ \|u\|_{L^\infty([0,1]\times\T)}^\nu\big) \Big(\int_0^{t-h}(t-s)^{-3/4}n^{-1}\,ds+\int_{t-h}^t 1\,ds\Big)
\\
&\preceq n^{-1},
\label{eq:bound-cE2}
\end{equs}
using Lemma \ref{eq:Pn-P} (with $\beta=2$) to get the third line, along with the fact that $\|p_{t-s}(x,\cdot)\|_{L^1(\T)}=\|p^n_{\kappa_n(t-s)}(x,\rho_n(\cdot))\|_{L^1(\T)}=\|p^n_{\kappa_n(t-s)}(x,\cdot)\|_{L^1(\Pi_n)}=1$.

To estimate $\cE^{n,1}$, introduce a parameter $L\in\N$ and the notations
$\tau_L=\inf\{t\geq 0:\,\|u_t\|_{L^\infty(\T)}\geq L\}\wedge 1$,
$\tilde u^{(L)}_s=u_{s\wedge\tau_L}$,
$\tilde \Omega_L=\{\tau_L=1\}\subset\{\|u\|_{L^\infty([0,1]\times\T)}\leq L\}$,
$F_L(r)=F\big(-L\vee(r\wedge L)\big)$.
Finally, let us introduce, for any $z:\Omega\times[0,1]\times\T\to\R$,
\begin{equ}
\tilde\cE^{n,L}_t[z](x)=\int_0^t\int_{\T} p_{t-s}(x,y)\Big(F_L\big(z_s(y)\big)-F_L\big(z_{\kappa_n(s)}(\rho_n(y))\big)\Big)\,dy\,ds.
\end{equ}
Then by definition, on $\tilde\Omega_L$, one has $\cE^{n,1}=\tilde\cE^{n,L}[\tilde u^{(L)}]$.
Moreover, the union of $\tilde\Omega_L$-s over all $L\in\N$, is of full probability.
Therefore, in order to show
\begin{equ}\label{eq:bound-cE1}
|\cE^{n,1}_t(x)|\preceq n^{-1+2\eps},
\end{equ}
it suffices to show for any $L\in\N$
\begin{equ}\label{eq:bound-cEL}
\big|\tilde \cE^{n,L}_t[\tilde u^{(L)}](x)\big|\preceq_L n^{-1+2\eps}.
\end{equ}
By Lemma \ref{lem:integral}, we have that
\begin{equ}
\big\|\tilde \cE^{n,L}_t[w](x)\big\|_{L^p(\Omega)}\lesssim_{p,L} n^{-1+\eps},
\end{equ}
where $w$ is defined in \eqref{eq:w}. Therefore,
\begin{equs}
\E\sup_{n\in\N}n^{1-2\eps}\sup_{(t,x)\in(\Lambda_n\cap[0,1])\times\Pi_n}\big|\tilde \cE^{n,L}_t[w](x)\big|^p
&\leq\E\sum_{n\in\N}n^{p(1-2\eps)}\sum_{(t,x)\in(\Lambda_n\cap[0,1])\times\Pi_n}\big|\tilde \cE^{n,L}_t[w](x)\big|^p
\\
&\lesssim_{p,L}\sum_{n\in\N}n^{-p\eps}n^3.\label{eq:44}
\end{equs}
Choosing $p$ large enough, the last sum converges, and we get \eqref{eq:bound-cEL} with $w$ in place of $u^{(L)}$.
Since the laws of $w$ and $\tilde u^{(L)}$ are equivalent by Girsanov's theorem,  \eqref{eq:bound-cEL}, and consequently \eqref{eq:bound-cE1}, follows.

Turning to $\cE^{n,3}$, define the polynomial $\tilde F$ in two variables $a,b$ by
\begin{equ}
\tilde F(a,b)=\sum_{j=0}^{\nu-1}c_{j+1}\sum_{k=0}^j a^{j-k}b^k,
\end{equ}
which satisfies $F(a)-F(b)=(a-b)\tilde F(a,b)$.
We then write
\begin{equs}
\cE^{n,3}_t&=\int_0^t\cP^n_{\kappa_n(t-s)}\Big(\big(O_{\kappa_n(s)}-O^n_{\kappa_n(s)}\big)\tilde
F(u_{\kappa_n(s)},u^n_{\kappa_n(s)})\Big)\,ds
\\
&\qquad+
\int_0^t\cP^n_{\kappa_n(t-s)}\Big(\big(v_{\kappa_n(s)}-v^n_{\kappa_n(s)}\big)\tilde
F(u_{\kappa_n(s)},u^n_{\kappa_n(s)})\Big)\,ds
\\
&=:\cE^{n,3,1}_t+\cE^{n,3,2}_t.
\end{equs}
It's well-known that
$\sup_{s\in[0,1]}\|u_s\|_{C^{1/2-\eps}(\T)}\preceq 1$.
Therefore by Lemma \ref{lem:besov-holder}
\begin{equ}
\sup_{s\in[0,1]}\|\delta u_s\|_{B^{1/2-\eps}_{\infty,\infty}(\Pi_n)}\lesssim
\sup_{s\in[0,1]}\|\delta u_s\|_{C^{1/2-\eps}(\Pi_n)}\leq 
\sup_{s\in[0,1]}\|u_s\|_{C^{1/2-\eps}(\T)}\preceq 1.
\end{equ}
Similarly, one gets $\sup_{s\in[0,1]}\|\delta O_s\|_{B^{1/2-\eps}_{\infty,\infty}(\Pi_n)}\preceq 1$.
By Lemma \ref{lem:upperbound} and an analogous argument to \eqref{eq:44} we have
\begin{equ}\label{ddd}
\sup_{s\in[0,1]} \big\|\delta O_{\kappa_n(s)}-O^n_{\kappa_n(s)}\big\|_{B^{\bar\theta}_{\infty,\infty}(\Pi_n)}\preceq n^{-1/2+\bar\theta+\eps}
\end{equ}
for any $\bar \theta\in(-1/2,1/2)$.
Choosing $\bar \theta>0$ we get
\begin{equ}
\sup_{n\in\N} R^n\leq 1+\sup_{n\in\N}\max_{t\in\Lambda_n\cap[0,1]} \|O^n_t\|_{B^{\bar \theta}_{\infty,\infty}(\Pi_n)}\preceq 1.
\end{equ} 
For any $\mu$ as in Section \ref{sec:apriori},
on $\Omega_n$, Theorem \ref{thm:apriori} yields a bound of order $1$ in the sense of $\preceq$ for $v^n$ in $L^\mu(\Pi_n)\subset B^0_{\mu,\infty}(\Pi_n)$. Substituting this back in the equation, by Lemma \ref{lem:discrete-Schauder} one also gets a bound of order $1$ in $B^{2}_{\mu/\nu,\infty}(\Pi_n)\subset B^{1/2-\eps}_{\infty,\infty}(\Pi_n)$.
We conclude that on $\Omega_n$ we have
\begin{equ}
\sup_{n\in\N}\sup_{s\in[0,1]}\big\|\tilde F(\delta u_{\kappa_n(s)},u^n_{\kappa_n(s)})\big\|_{B^{1/2-\eps}_{\infty,\infty}(\Pi_n)}\preceq 1.\label{eq:66}
\end{equ}
Now we use \eqref{ddd} with $\bar\theta=\theta$.
Since $1/2+\theta-\eps>0$, Lemma \ref{lem:product} can be used to bound the integrand in $\cE^{n,3,1}$, which combined with Lemma \ref{lem:discrete-Schauder} yields that on $\Omega_n$
\begin{equ}\label{eq:bound-cE31}
\sup_{t\in\Lambda_n\cap[0,1]}\|\cE^{n,3,1}_t\|_{B^{2+\theta}_{\infty,\infty}(\Pi_n)}\preceq  n^{-1/2+\theta+\eps}.
\end{equ}
Finally, we can use \eqref{eq:66} in a trivial way to get
for any $T\in\Lambda_n$, on $\Omega_n$,
\begin{equ}\label{eq:bound-cE32}
\sup_{t\in\Lambda_n\cap[0,T]}\|\cE^{n,3,2}_t\|_{L^\infty(\Pi_n)}\lesssim \int_0^T \|\delta v_{\kappa_n(s)}-v^n_{\kappa_n(s)}\|_{L^\infty(\Pi_n)}\,ds.
\end{equ}
Putting together \eqref{eq:CE-decomp}-\eqref{eq:bound-cE0}-\eqref{eq:bound-cE2}-\eqref{eq:bound-cE1}-\eqref{eq:bound-cE31}-\eqref{eq:bound-cE32} we get for any 
$T\in\Lambda_n$, on $\Omega_n$,
\begin{equ}
\sup_{t\in\Lambda_n\cap[0,T]} \|\delta v_t-v^n_t\|_{L^\infty(\Pi_n)}\preceq n^{-1/2+\theta+\eps}+\int_0^T \sup_{t\in\Lambda_n\cap[0,s]}\|\delta v_t-v^n_t\|_{L^\infty(\Pi_n)}\,ds.
\end{equ}
By Gronwall's lemma we get on $\Omega_n$
\begin{equ}
\sup_{t\in\Lambda_n\cap[0,1]} \|\delta v_t-v^n_t\|_{L^\infty(\Pi_n)}\preceq n^{-1/2+\theta+\eps},
\end{equ}
and since the sum of the probabilities of the complements of $\Omega_n$ is finite, this yields the claim.\qed
\newpage

\section{Implementation}
The implementation of the finite difference scheme is simple. The results are illustrated below.
\begin{center}
\begin{figure}[hbt!]
\centering
  \includegraphics[width=9cm]{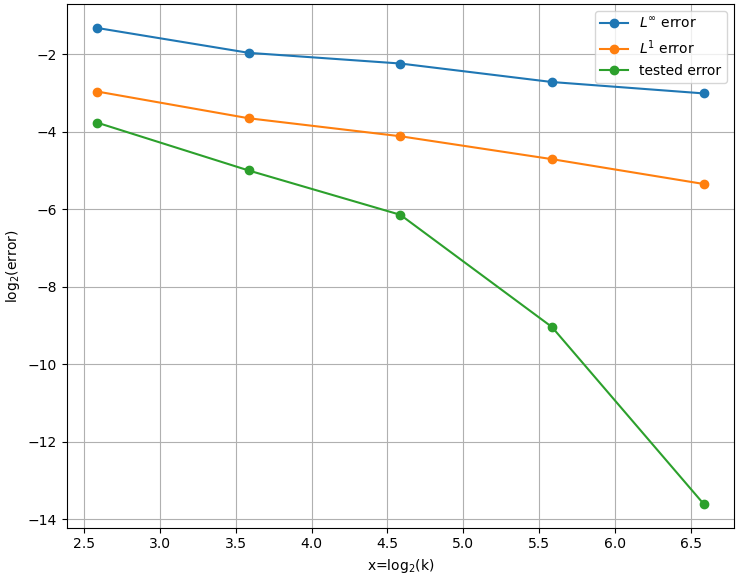}
  \caption{A log-log plot of the error (compared to the highest resolution approximation) for a single realisation of the approximations, measured in $L_\infty$, $L_1$, and tested against a single test function, showing superior rate for the latter.}  
  \label{figureC}
\end{figure}
\end{center}
\begin{figure}[hbt!]
  \includegraphics[width=\linewidth]{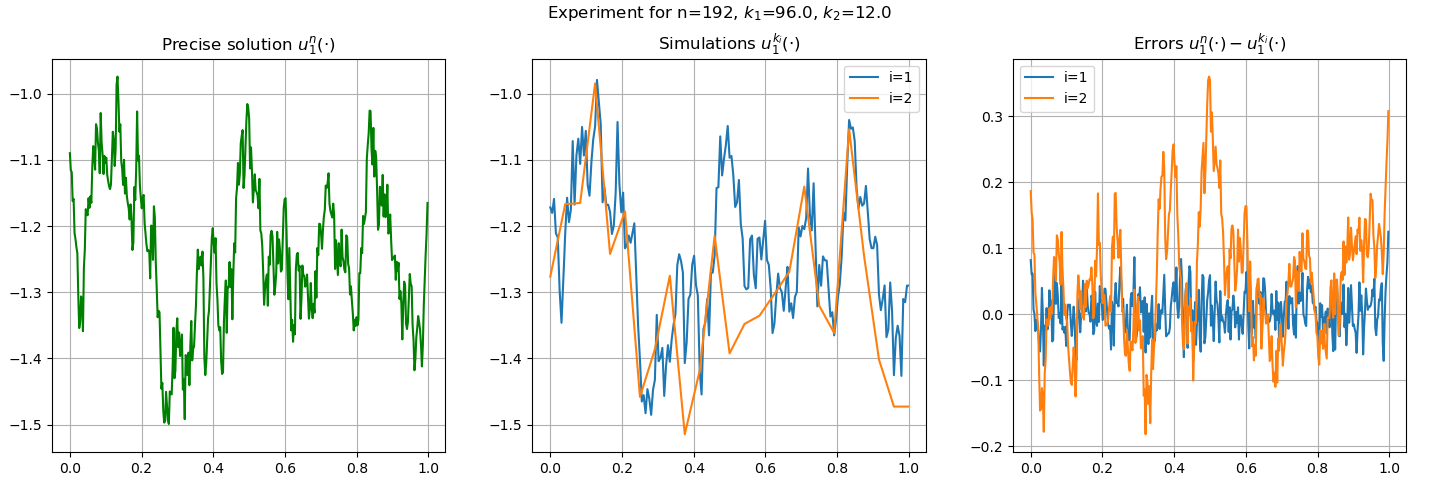}
\caption{A fine grid approximation at time $1$ (left), two coarser grid approximations (middle) and their differences (right). The error plots show oscillatory behaviour around $0$, which suggests smaller error in negative regularity spaces, agreeing with Theorem \ref{thm:main}.}  
  \label{figureA}
\end{figure}
\begin{figure}[hbt!]
  \includegraphics[width=\linewidth]{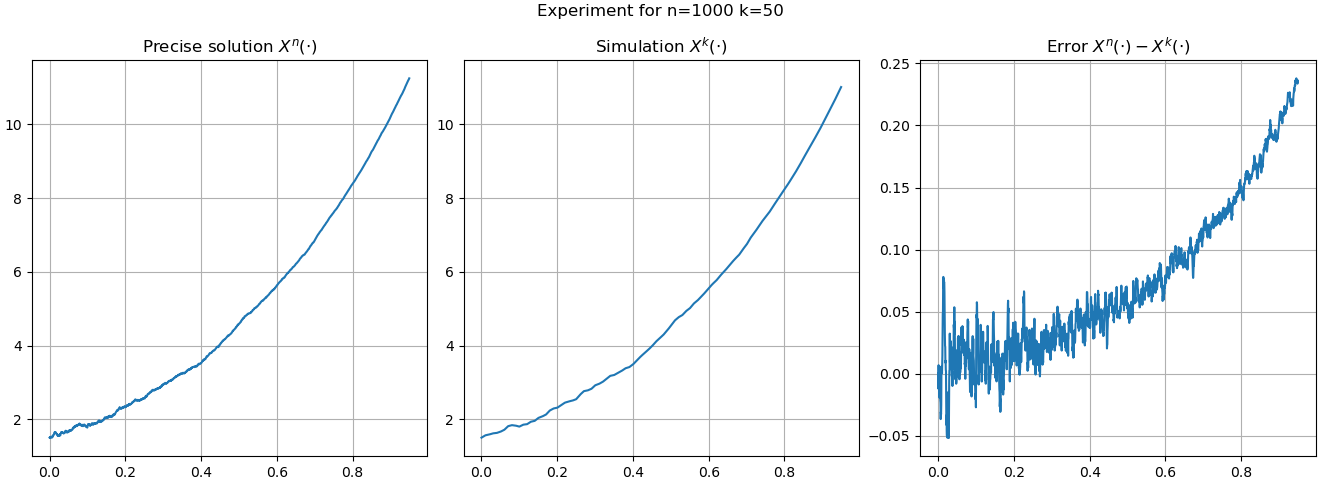}
  \caption{A fine grid approximation (left), a coarse grid approximation (middle) and their difference (right) for the Euler scheme applied to the SDE $dX_t=2X_t\,dt+\big(\arctan(X_t)-\pi/2\big)\,dW_t$.
Here the error is one-sided and thus not expected to be smaller in distributional norm, as discussed in Remark \ref{rem:SDE}.}  
  \label{figureB}
\end{figure}

\newpage
{\appendix
\section{Some technical proofs}
This appendix contains some proofs postponed from Section \ref{sec:spaces}.
\begin{proof}[Proof of Lemma~\ref{lemma:testing against test functions}]
Set $J_\lambda\in \{0,...,J_n+1\}$ such that $\lambda\in (2^{-J_\lambda-1},2^{-J_\lambda}]$, we find
\begin{align*}
\big|\scal{f,\delta\varphi_x^\lambda}_n\big|
&\leq\sum_{k=0}^{J_n+1} \big|\scal{f^{[k],n},\delta\varphi_x^\lambda}_n\big| \\
&\leq \sum_{j=0}^{J_n+1} \sum_{|k-j|\leq 1} \big|\scal{f^{[k],n},(\delta\varphi_x^\lambda)^{[j],n}}_n\big|
\\
&\leq \sum_{j=0}^{J_n+1} \sum_{|k-j|\leq 1} \big\|f^{[k],n}\big\|_{L^\infty(\Pi_n)}\big\|(\delta\varphi_x^\lambda)^{[j],n}\big\|_{L^1(\Pi_n)}\\
&\lesssim  \|f\|_{B^{\alpha}_{\infty,\infty}(\Pi_n)} \sum_{j=0}^{J_n+1} 2^{-\alpha j}\big\|(\delta\varphi_x^\lambda)^{[j],n}\big\|_{L^1(\Pi_n)} \\
&\leq\|f\|_{B^{\alpha}_{\infty,\infty}(\Pi_n)} \Big( \sum_{j=0}^{J_\lambda} 2^{-\alpha j}\big\|(\delta\varphi_x^\lambda)^{[j],n}\big\|_{L^1(\Pi_n)}+ \sum_{j=J_\lambda+1}^{J_n+1} 2^{-\alpha j}\big\|(\delta\varphi_x^\lambda)^{[j],n}\big\|_{L^1(\Pi_n)} \Big)\ .
\end{align*}
For the first sum, using \eqref{aaaa} and \eqref{eq:disc-to-cont-fourier}, we find 
$$\sum_{j=0}^{J_\lambda} 2^{-\alpha j} \big\|(\delta\varphi_x^\lambda)^{[j],n}\big\|_{L^1(\Pi_n)} \lesssim \|\varphi^\lambda_x\|_{L^1(\mathbb{R})} \sum_{j=0}^{J_\lambda}  2^{-\alpha j} \lesssim\|\varphi\|_{L^1(\T)}  \lambda^\alpha \ ,$$
where we used that $\|\varphi\|_{L^1(\T)}=  \|\varphi^\lambda_x\|_{L^1(\mathbb{R})}$.
For the second sum, we claim that for $2^{-j}\leq \lambda$ and any $l\in \mathbb{N}$
\begin{equation}\label{asdf}
\|(\delta\varphi_x^\lambda)^{[j],n}\|_{L^1(\Pi_n)}\lesssim_{n} \|\varphi\|_{C^{n}(\T)} (2^{-j} \lambda^{-1})^l,
\end{equation}
from which it follows that since $r>-\alpha$,
$$
\sum_{j=J_\lambda+1}^{J_n+1} 2^{-\alpha j}\|(\delta\varphi_x^\lambda)^{[j],n}\|_{L^1(\Pi_n)}\lesssim_{r}\|\varphi\|_{C^{r}(\T)} \sum_{j=J_\lambda+1}^{J_n+1}
 (2^{-j} \lambda^{-1})^r   2^{-\alpha j} \lesssim\|\varphi\|_{C^{r}(\T)}  \lambda^\alpha.
 $$
We show \eqref{asdf} for $j< J_n+1$, for the case $j= J_n+1$ one adapts the argument as usual. Since both sides of \eqref{asdf} are translation invariant, we can without loss of generality assume that $x=0$ and drop $x$ from the notation. We find
\begin{align*}
\big\|(\delta\varphi^\lambda)^{[j],n}\big\|_{L^1(\Pi_n)} &= \big\|\sum_{k\in\Z} \langle\varphi^\lambda, e_k\rangle \phi^j_{\rho_n}(k) \delta e_k\big\|_{L^1(\Pi_n)}\\
&\leq \big\|\sum_{k\in\Z} \langle\varphi^\lambda, e_k\rangle \phi^{j}_{\rho_n}(k)  e_k\big\|_{L^1(\T)}\\
&\leq \big\|\varphi^\lambda * \mathcal{F}(\phi^{j}_{\rho_n}) \big\|_{L^1(\R)} \ .
\end{align*}
Recall that $\mathcal{F}(\phi^{j}_{\rho_n})$ annihilates polynomials (which is easy to see on the Fourier side, using that $\phi^{j}_{\rho_n}$ is supported away from $0$). Therefore, denoting by ${\Poly}^r_z[\varphi^\lambda]$ the Taylor polynomial of $\varphi_x^\lambda$ of degree $r-1$ at $z$, we find
\begin{align*}
\big|\varphi^\lambda * \mathcal{F}(\phi^{j}_{\rho_n}) (z)\big|
&=\big|\int_{\mathbb{R}} \varphi^\lambda(y)  \mathcal{F}(\phi^{j}_{\rho_n})(z-y) \big)\,dy\big|\\
&= \big|\int_{\mathbb{R}}\big( \varphi^\lambda(y)- \Poly^r_z[\varphi^\lambda](y) ) \mathcal{F}(\phi^{j}_{\rho_n})(z-y) \big)\,dy\big|\\
&\lesssim  \lambda^{-1}\|\varphi\|_{C^{r}(\mathbb{R})} \int_{\mathbb{R}}\left| \frac{z-y}{\lambda}\right|^r \big|\mathcal{F}(\phi^{j}_{\rho_n})(z-y)\,dy \big|\\
&\lesssim \lambda^{-1} \|\varphi\|_{C^{r}(\T)}  (\lambda^{-1}2^{-j} )^r \ .
\end{align*}
Thus, $$\int_{{\{z\ : \ |z|<2\lambda\}}} |\varphi^\lambda * \mathcal{F}(\phi^{j}_{\rho_n}) (z)| \,dz\lesssim (\lambda^{-1}2^{-j} )^{r} \ .$$
Next, note that for any $m$ one has $|\mathcal{F}(\phi^{1}_{\rho_n})(y)| \lesssim_m \left|\frac{1}{z}\right|^m$. Choosing $m=r+1$, this implies that
\begin{align*}
\int_{{\{z\ : \ |z|>2\lambda\}}} |\varphi^\lambda * \mathcal{F}(\phi^{j}_{\rho_n}) (z)|\, dz 
&\lesssim \|\varphi^\lambda\|_{L^1(\mathbb{R})} \int_{{\{z\ : \ |z|>\lambda\}}}  \frac{1}{2^{-j}} \left|\frac{2^{-j}}{z} \right| ^{r+1} dz \\
&\lesssim \|\varphi\|_{L^1(\mathbb{R})} 2^{-rj} \int_{{\{z\ : \ |z|>\lambda\}}} \left|\frac{1}{z}\right|^{r+1} dz\\
&\lesssim \|\varphi\|_{L^1(\mathbb{R})} 2^{-rj} \lambda ^{-r}\\
&= \|\varphi\|_{L^1(\mathbb{R})} (\lambda^{-1} 2^{-j})^{r} \ ,
\end{align*}
completing the proof of \eqref{asdf}.
 \end{proof}
 
 \begin{proof}[Proof of Lemma \ref{lem:besov-holder}]
By Young's inequality
$$\|f^{[0],n}\|_{L^\infty(\Pi_n)}\leq \|\phi^0_{\rho_n}\|_{L^1(\Pi_n)} \|f\|_{L^\infty(\Pi_n)}\lesssim \|f\|_{C^\alpha(\Pi_n)} \ . $$
For $j\in \{1,...,J_n+1\}$ set $\Phi^{j,n}= \sum_{k\in\Z} \hat\phi^j_{\rho_n}(k)\delta e_k$.
Then it follows as in \eqref{aaaa} that 
$$f^{[j],n}= \Phi^{j,n}\ast_n f(x)= \frac{1}{2n}\sum_{y\in \Pi_n} \Phi^{j,n}(y)\big( f(x-y)- f(x)\big)$$
and thus 
$$\| f^{[j],n}\|_{L^\infty(\Pi_n)} \leq \big\| |\cdot|^\alpha |\Phi^{j,n}(\cdot)|\big\|_{L^1(\Pi_n)} \|f\|_{C^\alpha(\Pi_n)}\ .$$
to bound the last term we use Remark~\ref{rem:fourier}
$$\big\| |\cdot|^\alpha |\Phi^{j,n}(\cdot)|\big\|_{L^1(\Pi_n)}\leq \big\| |\cdot|^\alpha \mathcal{F}^{-1} (\hat\phi^j_{\rho_n})(\cdot)\big\|_{L^1((2n)^{-1}\mathbb{Z})}\ $$
and find that
\begin{align*}
\big\| |\cdot|^\alpha \mathcal{F}^{-1} (\hat\phi^j_{\rho_n})(\cdot)\big\|_{L^1((2n)^{-1}\mathbb{Z})}
&= \frac{1}{2n}\sum_{y\in (2n)^{-1}\mathbb{Z}} |y|^\alpha \frac{1}{2^{{-(j+1)}}}  \mathcal{F}^{-1} (\hat\phi^1_{\rho_n})\left(\frac{y}{2^{-(j+1)}}\right)\\
&=\frac{1}{2^{j+1}n}\sum_{w\in (2^{j+2}n)^{-1}\mathbb{Z}} \left|\frac{w}{2^{j+1}}\right|^\alpha \mathcal{F}^{-1}  (\hat\phi^1_{\rho_n})(w)\\
&=2^{-\alpha(j+1) }\big\| |\cdot|^\alpha \mathcal{F}^{-1} (\hat\phi^1_{\rho_n})(\cdot)\big\|_{L^1((2^{j+1}n)^{-1}\mathbb{Z})} \ , 
\end{align*}
which completes the first inequality in \eqref{eq:besov-holder}. To obtain the second inequality, given $f\in C(\Pi_n)$, we write $f_z(x)=f(x+z)$ for $z\in \Pi_n$.
Let $j_z\in \mathbb{N}$ be such that $|z|\in (2^{-(j_z+1)}, 2^{-j_z}]$, then 
\begin{align*}
\|f-f_z\|_{L^\infty(\Pi_n)} &\leq\sum_{j=0}^{j_z} \|f^{[j],n}-f^{[j],n}_z\|_{L^\infty(\Pi_n)} + \sum_{j=j_z}^{J_n+1} \big(\|f^{[j],n}\|_{L^\infty(\Pi_n)} +\|f^{[j],n}_z\|_{L^\infty(\Pi_n)}\big) \\
&\leq\sum_{j=0}^{j_z} \|f^{[j],n}-f^{[j],n}_z\|_{L^\infty(\Pi_n)} + 2\|f\|_{B^{\alpha}_{\infty,\infty}(\Pi_n)} \sum_{j=j_z}^{J_n+1} 2^{-\alpha j} \\
&\lesssim \sum_{j=0}^{j_z} \|f^{[j],n}-f^{[j],n}_z\|_{L^\infty(\Pi_n)} +\|f\|_{B^{\alpha}_{\infty,\infty}(\Pi_n)} |z|^{\alpha}.
\end{align*}
To bound the first sum, note that using the discrete derivative $D_n f(x):= \frac{f(x)-f(x+(2n)^{-1})}{(2n)^{-1}}$ one finds
\begin{equs}
\|f^{[j],n}-f^{[j],n}_z\|_{L^\infty(\Pi_n)}\lesssim |z|\| D_n f^{[j],n}\|_{L^\infty(\Pi_n)}&\lesssim |z| 2^{j(-\alpha+1)} \| D_n f\|_{B^{\alpha-1}_{\infty,\infty}(\Pi_n)}
\\
&\lesssim
|z| 2^{j(-\alpha+1)} \| f\|_{B^{\alpha}_{\infty,\infty}(\Pi_n)} \ , 
\end{equs}
where we used Lemma~\ref{derivative is continuous lemma needed} from below in the last inequality. Thus one completes the proof by observing that since $\alpha \in (0,1)$, one has
$$\sum_{j=0}^{j_z} \|f^{[j],n}-f^{[j],n}_z\|_{L^\infty(\Pi_n)} \lesssim |z| \| f\|_{B^{\alpha}_{\infty,\infty}(\Pi_n)} \sum_{j=0}^{j_z} 2^{j(-\alpha+1)} \lesssim \| f\|_{B^{\alpha}_{\infty,\infty}(\Pi_n)} |z|^\alpha.$$
\end{proof}
\begin{lemma}\label{derivative is continuous lemma needed}
For any $\alpha\in\mathbb{R}$ there exists a constant $N=N(\phi^0)$ such that for all $n\in\N$, $f\in C(\Pi_n)$ one has the bounds
$$\|D_nf\|_{B^{\alpha-1}_{p,q}(\Pi_n)}\leq N\|f\|_{B^{\alpha}_{p,q}(\Pi_n)} \ , $$
where $D_n f(x):= \frac{f(x)-f(x+(2n)^{-1})}{(2n)^{-1}}$. 
\end{lemma}
\begin{proof}
We prove that $\| (D_nf)^{[j],n}\|_{L^p(\Pi_n)} \lesssim 2^j \| f^{[j],n}\|_{L^p(\Pi_n)}$ for $j\in \{1,\ldots,J_n\}$, the claim for $j=J_n+1$ follows by the usual adaptation of the same argument, while the case $j=0$ is trivial.
Let $\eta^1_{1}: \R\to [0,1]$ be a smooth compactly supported function such that $\eta^1_1|_{\supp \phi^1_1}\equiv 1$, $0\notin \supp \eta^1_1$, and $\supp\eta^1_1\subset B_{(1+\eps_0)\rho_0^{-1}}$. Then set $\eta^j_\rho(x):= \eta(2^{-j}\rho^{-1}x)$ and observe that for $j\geq 1$
$$
(D_nf)^{[j],n}= D_n \Big(\big(\sum_{k=-n}^{n-1} \eta^j_{\rho_n}(k)  \delta e_k \big)*_n f^{[j],n}\Big)= D_n \big(\sum_{k=-n}^{n-1} \eta^j_{\rho_n}(k) \delta e_k \big) *_n f^{[j],n}
$$
and thus 
$$\| (D_nf)^{[j],n}\|_{L^p(\Pi_n)}\leq \Big\|D_n \big(\sum_{k=-n}^{n-1} \eta^j_{\rho_n}(k) \delta e_k \big)\Big\|_{L^1(\Pi_n)} \|f^{[j],n}\|_{L^p(\Pi_n)}$$
which completes the proof, since
\begin{align*}
\Big\|D_n \big(\sum_{k=-n}^{n-1} \eta^j_{\rho_n}(k) \delta e_k \big)\Big\|_{L^1(\Pi_n)} 
&=\Big\|D_n \big(\sum_{k\in \mathbb{Z}}{\eta}^j_{\rho_n}(k) \delta e_k \big)\Big\|_{L^1(\Pi_n)} \\
& \lesssim \frac{1}{2n}\big\|(\mathcal{F}^{-1}{\eta}^j_{\rho_n})(\cdot)- (\mathcal{F}^{-1}{\eta}^j_{\rho_n})(\cdot-(2n)^{-1})\big\|_{L^1(\mathbb{R})}\\
& = \frac{1}{2n}\Big\|\int_0^{(2n)^{-1}}(\partial \mathcal{F}^{-1}{\eta}^j_{\rho_n})(\cdot+y)dy\Big\|_{L^1(\mathbb{R})}\\
& \leq \|\partial \mathcal{F}^{-1}{\eta}^j_{\rho_n}\|_{L^1(\mathbb{R})}\\
&=2^{j} \|\partial \mathcal{F}^{-1}{\eta}^1_{\rho_n}\|_{L^1(\mathbb{R})}\\
&\lesssim 2^{j} \|\partial \mathcal{F}^{-1}{\eta}^1\|_{L^1(\mathbb{R})}.
\end{align*}
\end{proof}
}

\smallskip
\noindent\textbf{Acknowledgments}
MG was funded by the Austrian Science Fund (FWF) Stand-Alone programme P 34992.
HS acknowledges funding by the Imperial College London President's PhD Scholarship.

\bibliography{AllenCahn}{}
\bibliographystyle{Martin}

\end{document}